\documentclass[a4paper,11pt,reqno]{amsart}

\usepackage{amsthm,amsmath,amsfonts,amssymb}
\usepackage{color}

\usepackage{cases}

\usepackage{cite}
\usepackage[colorlinks=true,urlcolor=blue,
citecolor=red,linkcolor=blue]{hyperref}

\usepackage{graphicx}

\usepackage{pgf, tikz-cd}
\usetikzlibrary{arrows}
\usepackage{mathrsfs}
\usepackage{tikzscale}
\usepackage{subfigure}

\numberwithin{equation}{section}
\newcommand{\e}{\varepsilon}

\renewcommand{\k}{\kappa}

\newcommand{\g}{\gamma}
\newcommand{\G}{\Gamma}
\newcommand{\s}{\sigma}
\renewcommand{\t}{\tau}

\renewcommand{\phi}{\varphi}

\newcommand{\N}{\mathbb{N}}
\renewcommand{\S}{\mathbb{S}}
\def\R{\mathbb{R}}

\def\H{\mathcal{H}}

\def\cF{\mathcal{F}}

\DeclareMathOperator{\dist}{dist}

\DeclareMathOperator{\interior}{int}

\DeclareMathOperator{\reach}{reach}

\DeclareMathOperator{\inr}{inr}

\newcommand{\de}{\partial}
\newcommand{\Om}{\Omega}

\theoremstyle{plain}
\newtheorem{thm}{Theorem}[section]
\newtheorem*{thm*}{Theorem}
\newtheorem{lem}[thm]{Lemma}
\newtheorem{prop}[thm]{Proposition}
\newtheorem{cor}[thm]{Corollary}

\theoremstyle{definition}
\newtheorem{defin}[thm]{Definition}

\theoremstyle{remark}
\newtheorem{rem}[thm]{Remark}

\numberwithin{equation}{section}

\title[Prescribed curvature minimizers in a Jordan domain]{Minimizers of the prescribed curvature functional in a Jordan domain with no necks}

\author{Gian Paolo Leonardi}
\address{Dipartimento di Matematica, via Sommarive 14, IT-38123 Povo - Trento (Italy)}
\email{gianpaolo.leonardi@unitn.it}
\author{Giorgio Saracco}
\address{Dipartimento di Matematica, via Ferrata 5, IT-27100 Pavia - Pavia (Italy)}
\email{giorgio.saracco@unipv.it}

\thanks{G.~P.~L.~and G.~S.~have been partially supported by the INdAM--GNAMPA Project 2019 ``Problemi isoperimetrici in spazi Euclidei e non'' (n.~prot.~U-UFMBAZ-2019-000473 11-03-2019).}



\begin{document}

\definecolor{ffffff}{rgb}{1.,1.,1.}
\definecolor{cqcqcq}{rgb}{0.75,0.75,0.75}

\begin{abstract}
We provide a geometric characterization of the minimal and maximal minimizer of the prescribed curvature functional $P(E)-\k |E|$ among subsets of a Jordan domain $\Om$ with no necks of radius $\k^{-1}$, for values of $\k$ greater than or equal to the Cheeger constant of $\Om$.  As an application, we describe all minimizers of the isoperimetric profile for volumes greater than the volume of the minimal Cheeger set, relative to a Jordan domain $\Om$ which has no necks of radius $r$, for all $r$. Finally, we show that for such sets and volumes the isoperimetric profile is convex.
\end{abstract}

 \hspace{-3cm}
 {
 \begin{minipage}[t]{0.6\linewidth}
 \begin{scriptsize}
 \vspace{-3cm}
 This is a pre-print of an article published in \emph{ESAIM Control Optim. Calc. Var.}. The final authenticated version is available online at: \href{https://doi.org/10.1051/cocv/2020030}{https://doi.org/10.1051/cocv/2020030}
 \end{scriptsize}
\end{minipage} 
}

\maketitle

\section{Introduction}

The existence and the study of properties of hypersurfaces in $\R^n$, with mean curvature given by some prescribed function $g\colon \R^n \to \R$, are classical problems in geometric analysis and in Calculus of Variations, see e.g.~\cite{Gia73, Gia74, Giu78, GN12, GMT81, GR10, Hu91, Mas74, MR10, TW83, Yau82} and the references therein. In the setting of oriented boundaries, the variational approach to the prescribed mean curvature problem is based on the minimization of the functional
\begin{equation}\label{eq:pmc_g}
\cF_g[F]= P(F) - \int_F g\, \textrm{d}x,
\end{equation}
where $P(F)=P(F; \R^n)$ is the total perimeter, intended in the $BV$ framework (see~\cite{AFP00book, Mag12book}). The function $g$ that shows up in~\eqref{eq:pmc_g} plays the role of a prescribed mean curvature, in the sense that any smooth critical point $F$ for $\cF_g$ satisfies $H_F(x) = g(x)$ at any $x\in \de F$, where $H_F(x)$ is the mean curvature of $\de F$ at $x$. A nice introduction to the problem in $\R^2$ and $\R^3$ is available in~\cite{BCG02}. When $g\ge 0$, the minimization of the functional~\eqref{eq:pmc_g} is tied to the weighted isoperimetric problem with volume density given by $g$: any minimizer $E$ of~\eqref{eq:pmc_g} is as well a perimeter minimizer among all sets $F$ that have the same ``weighted volume'' of $E$, i.e.~$\int_F g = \int_E g$. Some results in this setting have been obtained for instance in~\cite{ABCMP17, ABCMP19a, PS19, PS18} with in mind applications such as Hardy--Sobolev inequalities~\cite{CRO13, Csa15}, capillarity~\cite{FG79a, FG79b, Giu78, LS18a}, and even politics~\cite{FLSS18, SS19}.\par
In this paper we are interested in studying the structure of minimizers of~\eqref{eq:pmc_g} when $g$ is a positive constant, among subsets of an open, bounded set $\Om \subset \R^2$. Specifically, for a given positive constant $\k$ we consider the minimization of the functional
\begin{equation}\label{eq:pmc}
\cF_\k[F]= P(F) - \k|F|
\end{equation}
among measurable sets $F\subset \Om$, where $|\cdot|$ denotes the $2$-dimensional Lebesgue measure. 
It is well known that the internal boundary $\de E_\k \cap \Om$ of any nontrivial minimizer $E_\k$ of~\eqref{eq:pmc} is smooth and made of an at most countable union of circular arcs with curvature equal to $\k$. Existence of minimizers of~\eqref{eq:pmc} follows from the Direct Method of the Calculus of Variations, see~\cite[Section~12.5]{Mag12book}, but it may happen that the minimum is achieved by the empty set. A special value of $\k$ is given by the \emph{Cheeger constant} of $\Om$, defined as
\[
h_\Om= \inf \left\{\, \frac{P(F)}{|F|}: |F|>0\,, F\subseteq \Om  \,\right\},
\]
and any nontrivial set $E$ attaining the infimum is called \emph{Cheeger set} of $\Om$. The computation of the constant $h_\Om$ and the characterization of the Cheeger sets of $\Om$ are referred to as the \emph{Cheeger problem}. The existence of Cheeger sets is well known, see for instance~\cite{Leo15, Par11, PS17, Sar18}. Clearly, any Cheeger set $E$ is a nontrivial minimizer of~\eqref{eq:pmc} for the choice $\k=h_\Om$, i.e.~of
\begin{align*}
\cF_{h_\Om}[F]= P(F) - h_\Om|F|.
\end{align*}
Notice that $\min \cF_{h_\Om} = 0$ and that $\min \cF_{\k} \le \cF_\k[\emptyset]=0$, for all $\k>0$. On the one hand, if $\k> h_\Om$, one has
\[
\min \cF_\k \le P(E)-\k|E| < P(E) -h_\Om|E| = 0,
\]
where $E$ is a Cheeger set of $\Om$; this shows that $\cF_\k$ admits nontrivial minimizers. On the other hand, if $\min \cF_\k \ge 0$, then $P(F)|F|^{-1} \ge \k$ for all subset $F\subseteq \Om$ such that  $|F|>0$, hence by taking the infimum one finds that $\k \le h_\Om$. Therefore, the unique minimizer of~\eqref{eq:pmc} whenever the strict inequality $\k<h_\Om$ holds is the empty set. In the equality case, both the empty set and the Cheeger sets of $\Om$ solve~\eqref{eq:pmc}; in this limiting case, we shall always consider the nontrivial minimizers.\par

The Cheeger problem has been widely studied in the past, due to its deep connections with other problems ranging from eigenvalue estimates to capillarity. Several authors addressed the question about how to characterize and efficiently compute the value of the Cheeger constant $h_\Om$. The known results in this direction are essentially limited to the planar setting, as they heavily rely on the rigid characterization of curves with constant curvature in the plane. In particular, under the assumption that $\Om$ is convex~\cite{KL06} or a strip~\cite{LP16} it has been proved that the Cheeger set of $\Om$ is unique and precisely characterized from the geometric viewpoint. If we denote by $\Om^r$ the \emph{inner parallel set} at distance $r$, i.e.
\[
\Om^r = \{\,x\in \Om: \dist(x; \de \Om) \ge r\,\}\,,
\]
then the unique Cheeger set $E$ of $\Om$ is given by the Minkowski sum $\Om^r\oplus B_r$, where $r=h_\Om^{-1}$. Equivalently, the Cheeger set $E$ agrees with the union of all balls of radius $r$ contained in $\Om$. Moreover, the \emph{inner Cheeger formula} holds, i.e.~the radius $r$ is the unique positive solution of the equation 
\[
\pi \rho^2 = |\Om^\rho|.
\]
This formula and this kind of structure for planar Cheeger sets have been recently extended in~\cite{LNS17} to a class of planar domains that is essentially the largest possible. Before recalling the statement of the general structure theorem, we need to introduce the following definition of \emph{no necks of radius $r$} for $r\in (0, \inr(\Om)]$, where $\inr(\Omega)$ stands for the \emph{inradius} of $\Omega$. 

\begin{defin}\label{def:no_necks}
A set $\Om$ has \emph{no necks of radius $r$}, with $r\in(0, \inr(\Om)]$ if the following condition holds. If $B_r(x_0)$ and $B_r(x_1)$ are two balls of radius $r$ contained in $\Om$, then there exists a continuous curve $\g\colon [0,1]\to \Om$ such that 
\[
\g(0)=x_0, \qquad \g(1)=x_1, \qquad B_r(\g(t)) \subset \Om,\quad \forall t\in [0,1].
\]
\end{defin}

We remark that having no necks of radius $r_1$ does not imply the same property for any radius $r_2<r_1$.\par

Whenever a set has no necks of radius $r=h_\Om^{-1}$, then its (maximal) Cheeger set agrees with the union of all balls of radius $r$ contained in $\Om$, analogously to what happens for convex sets and strips. This remarkable fact was proved in~\cite{LNS17}, and we recall the theorem below.

\begin{thm}[Theorem~1.4 and Remark~5.2 of~\cite{LNS17}]\label{thm:LNS}
Let $\Om$ be a Jordan domain such that  $|\de \Om|=0$. If $\Om$ has no necks of radius $r=h_\Om^{-1}$, then the maximal Cheeger set $E$ of $\Om$ is given by
\[
E = \Om^r\oplus B_r\,,
\]
i.e.~the Minkowski sum of $\Om^r$ and $B_r$. Moreover, $r$ is the unique positive solution of
\begin{equation}\label{eq:inner_Cheeger_formula}
\pi \rho^2 = |\Om^\rho|\,.
\end{equation}
Finally, if $\Om^r = \overline{\interior{\Om^r}}$, then $E$ is the unique Cheeger set of $\Om$.
\end{thm}

We remark that $|\de \Om|$ is the $2$-dimensional Lebesgue measure of $\de \Om$, thus sets whose boundary is a plane-filling curve \`a la Knopp--Osgood (see~\cite{Sag94}) are not covered by the theorem. While it is unclear whether the hypothesis $|\de \Om|=0$ is necessary, the other hypothesis of topological flavor, i.e.~that $\Om$ is a Jordan domain, and the assumption of no necks of radius $h^{-1}_\Om$, must be required, otherwise one can produce counterexamples (see~\cite{LS18b, LNS17}). While uniqueness is not always granted in this more general setting, one can speak of \textit{the} 
\emph{maximal} Cheeger set because the class of Cheeger sets closed under countable unions: one can define a maximal Cheeger set (see Definition~\ref{def:min_max}) and prove its uniqueness (see Proposition~\ref{prop:min_max}).\par

In this paper we show that an analogous result to Theorem~\ref{thm:LNS} holds for nontrivial minimizers of the prescribed curvature functional $\cF_\k$. Specifically, in Theorem~\ref{thm:main}  we show that if a Jordan domain $\Om$ with $|\de\Om|=0$ has no necks of radius $r=\k^{-1}$, then the maximal minimizer $E^M_\k$ of $\cF_\k$ is given by $E^M_\k = \Om^r\oplus B_r$. Moreover, thanks to a careful study of the set $\Om^r \setminus \overline{\interior(\Om^r)}$, see Proposition~\ref{prop:struttura_diff}, we are able to give a precise geometric description of the unique minimal minimizer $E^m_\k$ of $\cF_\k$ and therefore to completely characterize the cases when uniqueness is granted (for the definition of minimal minimizer, we refer the reader to Definition~\ref{def:min_max}).\par

Once these characterizations are proved, we are able to describe \emph{all} possible minimizers of $\cF_\k$ by suitably ``interpolating'' between $E^m_\k$ and $E^M_\k$, and consequently we show that there exists a minimizer $E_\k$ of $\cF_\k$ such that $|E_\k|=V$, for any prescribed volume $V$ between $|E^m_\k|$ and $|E^M_\k|$. In Theorem~\ref{thm:m_iso} we apply this fact to the isoperimetric problem in a Jordan domain $\Om$ with $|\de \Om|=0$ that has no necks of radius $r$, for all $r\le h^{-1}_\Om$. For such an $\Om$ we can fully describe the isoperimetric sets relative to volumes $V\ge |E^m_{h_\Om}|$, and we show that the isoperimetric profile is convex in the volume range $|E^m_{h_\Om}|\le V \le |\Om|$.\par

The paper is structured as follows. In Section~\ref{sec:results} we state our main results and comment them. In Section~\ref{sec:properties} we state some properties of minimizers of~\eqref{eq:pmc} which are well known in the limit case $\k=h_\Om$, and whose extensions to any $\k \ge h_\Om$ are mostly trivial. In Section~\ref{sec:prelim} we give a characterization of the set difference $\Om^r \setminus \overline{\interior(\Om^r)}$, when $\Om$ has no necks of radius $r$. In Section~\ref{sec:main} we prove the structure of the maximal and minimal minimizers of~\eqref{eq:pmc} for $\k$, whenever $\Om$ has no necks of radius $\k^{-1}$. In Section~\ref{sec:iso} we address the isoperimetric problem in sets  $\Om$ with no necks of radius $r$ for all $r\le h_\Om^{-1}$, proving the structure of minimizers with volume greater than a certain threshold and the convexity of the isoperimetric profile above such a threshold.

\section{Statement of the main results}\label{sec:results}

Throughout the paper, with a slight abuse of notation, given a curve $\g\colon[0,1] \to \R^2$, we shall write $\g$ in place of $\g([0,1])$. For the sake of completeness, we recall that a Jordan domain is the region bounded by an injective and continuous map $\Phi\colon \mathbb{S}^1\to \mathbb{R}^2$, which is well defined thanks to the Jordan--Schoenflies theorem.\par

The first result we are going to prove is a characterization of the set difference $\Om^r \setminus \overline{\interior(\Om^r)}$, whenever $\Om$ is a Jordan domain with no necks of radius $r$. This, roughly speaking, says that such a difference consists of two families of curves $\Gamma_r^1$ and $\Gamma_r^2$:  curves in $\Gamma^1_r$ correspond to the presence of ``tendrils'' of width $r$, while curves in $\Gamma^2_r$ to the presence of ``handles'' of width $r$ as shown in Figure~\ref{fig:g1_g2}.

\begin{prop}\label{prop:struttura_diff}
Let $\Om$ be a Jordan domain with no necks of radius $r$. The following properties hold:
\begin{itemize}
\item[(a)] if $\Om^r$ is nonempty but has empty interior, then either it consists of a single point or there exists an embedding $\g\colon[0,1]\to \R^2$ of class $\textrm{C}^{1,1}$, with curvature bounded by $r^{-1}$, such that $\g([0,1]) = \Om^r$;
\item[(b)] if $\interior(\Om^r) \neq \emptyset$, then there exist two (possibly empty) families $\G^1_r$ and $\G^2_r$ of embedded curves contained in $\Om^r$ with the following properties. For each $i=1,2$ and each $\g\in \G^i_r$, 
\begin{itemize}
\item[(i)] $\g\colon[0,1]\to \Om^r$ is nonconstant and of class $\textrm{C}^{1,1}$, with curvature bounded by $r^{-1}$;
\item[(ii)] if $i=1$, then $\overline{\interior(\Om^r)}\cap \g = \{\g(0)\}$;
\item[(iii)] if $i=2$, then $\overline{\interior(\Om^r)}\cap \g = \{\g(0), \g(1)\}$;
\item[(iv)] $\G^1_r$ is finite;
\item[(v)] the following set equality holds
\[
\Om^r \setminus \overline{\interior(\Om^r)} = \bigcup_{\g\in \G^1_r} \g \left((0,1]\right) \cup \bigcup_{\g \in \G^2_r} \g \left((0,1)\right).
\]
\end{itemize}
\end{itemize}
\end{prop}

The structure granted by Proposition~\ref{prop:struttura_diff} might turn out useful in other contexts. We recall indeed, e.g.~the $\infty$-Laplacian problem~\cite{CF15} and the irrigation problem~\cite{BS03, Til10}, in which the set $\Om^r$ plays a role.

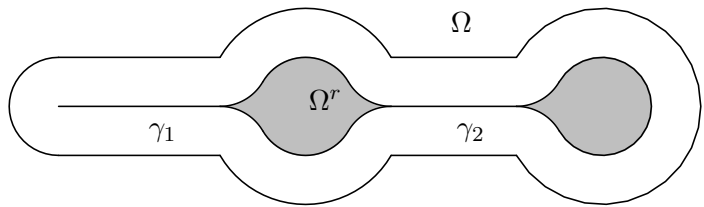
\begin{figure}
\begin{tikzpicture}[line cap=round,line join=round,>=triangle 45,x=1.0cm,y=1.0cm, scale=.65]
\clip(-6.2,-2.2) rectangle (8.2,2.2);
\fill[line width=0.pt,color=cqcqcq,fill=cqcqcq,fill opacity=1.0] (-1.7320508075688772,0.) -- (-0.866025403784438,0.5) -- (-0.8597787687708186,-0.5106666904850296) -- cycle;
\fill[line width=0.pt,color=cqcqcq,fill=cqcqcq,fill opacity=1.0] (1.7320508075688772,0.) -- (0.8660254037844382,0.5) -- (0.8597787687708185,-0.5106666904850297) -- cycle;
\fill[line width=0.pt,color=cqcqcq,fill=cqcqcq,fill opacity=1.0] (4.267949192431123,0.) -- (5.133974596215562,0.5) -- (5.140221231229182,-0.5106666904850299) -- cycle;
\draw [line width=0.6pt,color=cqcqcq,fill=cqcqcq,fill opacity=1.0] (6.,0.) circle (1.cm);
\draw [line width=0.6pt,color=ffffff,fill=ffffff,fill opacity=1.0] (4.267949192431123,1.) circle (1.cm);
\draw [line width=0.6pt,color=cqcqcq,fill=cqcqcq,fill opacity=1.0] (0.,0.) circle (1.cm);
\draw [line width=0.6pt,color=ffffff,fill=ffffff,fill opacity=1.0] (1.732050807568877,1.) circle (1.cm);
\draw [line width=0.6pt,color=ffffff,fill=ffffff,fill opacity=1.0] (-1.7320508075688767,1.) circle (1.cm);
\draw [line width=0.6pt,color=ffffff,fill=ffffff,fill opacity=1.0] (1.732050807568877,-1.) circle (1.cm);
\draw [line width=0.6pt,color=ffffff,fill=ffffff,fill opacity=1.0] (-1.7320508075688767,-1.) circle (1.cm);
\draw [line width=0.6pt,color=ffffff,fill=ffffff,fill opacity=1.0] (4.267949192431123,-1.) circle (1.cm);
\draw [line width=0.6pt] (-5.,1.)-- (-1.7320508075688772,1.);
\draw [line width=0.6pt] (1.7320508075688774,1.)-- (4.267949192431123,1.);
\draw [line width=0.6pt] (-5.,-1.)-- (-1.7320508075688772,-1.);
\draw [line width=0.6pt] (1.7320508075688774,-1.)-- (4.267949192431123,-1.);
\draw [shift={(-5.,0.)},line width=0.6pt]  plot[domain=1.5707963267948966:4.71238898038469,variable=\t]({1.*1.*cos(\t r)+0.*1.*sin(\t r)},{0.*1.*cos(\t r)+1.*1.*sin(\t r)});
\draw [shift={(0.,0.)},line width=0.6pt]  plot[domain=0.5235987755982988:2.6179938779914944,variable=\t]({1.*2.*cos(\t r)+0.*2.*sin(\t r)},{0.*2.*cos(\t r)+1.*2.*sin(\t r)});
\draw [shift={(0.,0.)},line width=0.6pt]  plot[domain=3.665191429188092:5.759586531581287,variable=\t]({1.*2.*cos(\t r)+0.*2.*sin(\t r)},{0.*2.*cos(\t r)+1.*2.*sin(\t r)});
\draw [shift={(-1.7320508075688772,1.)},line width=0.6pt]  plot[domain=4.71238898038469:5.759586531581288,variable=\t]({1.*1.*cos(\t r)+0.*1.*sin(\t r)},{0.*1.*cos(\t r)+1.*1.*sin(\t r)});
\draw [line width=0.6pt] (-5.,0.)-- (-1.7320508075688772,0.);
\draw [shift={(-1.732050807568877,-1.)},line width=0.6pt]  plot[domain=4.71238898038469:5.759586531581288,variable=\t]({1.*1.*cos(\t r)+0.*1.*sin(\t r)},{0.*1.*cos(\t r)+-1.*1.*sin(\t r)});
\draw [shift={(1.7320508075688774,1.)},line width=0.6pt]  plot[domain=4.71238898038469:5.759586531581288,variable=\t]({-1.*1.*cos(\t r)+0.*1.*sin(\t r)},{0.*1.*cos(\t r)+1.*1.*sin(\t r)});
\draw [shift={(1.7320508075688767,-1.)},line width=0.6pt]  plot[domain=4.71238898038469:5.759586531581288,variable=\t]({-1.*1.*cos(\t r)+0.*1.*sin(\t r)},{0.*1.*cos(\t r)+-1.*1.*sin(\t r)});
\draw [shift={(4.267949192431123,1.)},line width=0.6pt]  plot[domain=4.71238898038469:5.759586531581288,variable=\t]({1.*1.*cos(\t r)+0.*1.*sin(\t r)},{0.*1.*cos(\t r)+1.*1.*sin(\t r)});
\draw [shift={(4.267949192431123,-1.)},line width=0.6pt]  plot[domain=4.71238898038469:5.759586531581288,variable=\t]({1.*1.*cos(\t r)+0.*1.*sin(\t r)},{0.*1.*cos(\t r)+-1.*1.*sin(\t r)});
\draw [line width=0.6pt] (1.7320508075688772,0.)-- (4.267949192431123,0.);
\draw [shift={(6.,0.)},line width=0.6pt]  plot[domain=-2.617993877991494:2.617993877991494,variable=\t]({1.*1.*cos(\t r)+0.*1.*sin(\t r)},{0.*1.*cos(\t r)+1.*1.*sin(\t r)});
\draw [shift={(0.,0.)},line width=0.6pt]  plot[domain=3.6651914291880923:5.759586531581287,variable=\t]({1.*1.*cos(\t r)+0.*1.*sin(\t r)},{0.*1.*cos(\t r)+1.*1.*sin(\t r)});
\draw [shift={(0.,0.)},line width=0.6pt]  plot[domain=0.5235987755982994:2.617993877991494,variable=\t]({1.*1.*cos(\t r)+0.*1.*sin(\t r)},{0.*1.*cos(\t r)+1.*1.*sin(\t r)});
\draw [shift={(6.,0.)},line width=0.6pt]  plot[domain=-2.617993877991494:2.617993877991494,variable=\t]({1.*2.*cos(\t r)+0.*2.*sin(\t r)},{0.*2.*cos(\t r)+1.*2.*sin(\t r)});
\draw (2.85396543022974,-.15) node[anchor=north west] {$\gamma_2$};
\draw (-3.3815835625861905,-.15) node[anchor=north west] {$\gamma_1$};
\draw (2.7340510265217413,2.1279954111917614) node[anchor=north west] {$\Omega$};
\draw (-0.14389466247022642,0.5391295620607859) node[anchor=north west] {$\Omega^r$};
\end{tikzpicture}
\caption{Curves $\g_2$ with both endpoints in $\overline{\interior(\Om^r)}$ correspond to ``handles'' and connect disjoint connected components of $\interior(\Om^r)$, while curves $\g_1$ with just one endpoint in $\overline{\interior{(\Om^r)}}$ correspond to ``tendrils''.}\label{fig:g1_g2}
\end{figure}

\begin{defin}\label{def:min_max}
Let $\Om \subset \R^2$ and $\k>0$ be fixed, and let $E_\k$ be a minimizer of $\cF_\k$. We say that $E_\k$ is a \emph{maximal} minimizer if for any other minimizer $F_\k$ one has $F_\k \subset E_\k$; we say that it is a \emph{minimal} minimizer if for any other minimizer $F_\k$ one cannot have the strict inclusion $F_\k \subsetneq E_\k$.
\end{defin}

The existence of maximal and minimal minimizers is proved in Proposition~\ref{prop:min_max}, along with the uniqueness of the maximal minimizer. Concerning the uniqueness of minimal minimizers, it is verified when $\k>h_\Om$ but may fail in the case $\k=h_\Om$ (see again Proposition~\ref{prop:min_max} and Remark~\ref{rem:uniqueness_minimal}). In what follows we shall denote by $E_\k^M$ the maximal minimizer and by $E_\k^m$ the minimal minimizer in case the latter is unique.

\begin{thm}\label{thm:main} 
Let $\Om$ be a Jordan domain with $|\de \Om|=0$ and let $\k\ge h_\Om$ be fixed. Assume $\Om$ has no necks of radius $r=\k^{-1}$. Then, both maximal and minimal minimizers $E_\k^M$ and $E_\k^m$ are uniquely characterized as
\begin{align*}
E^M_\k = \Om^r\oplus B_r, && E^m_\k =\left(\overline{\interior(\Om^r)} \cup \bigcup_{\g\in\Gamma^2_r} \g \right) \oplus B_r.
\end{align*}
In particular, $\cF_\k$ has a unique minimizer (i.e., $E_\k^m=E_\k^M$) as soon as $\G^1_r$ is empty.
\end{thm}

Theorem~\ref{thm:main} extends Theorem~\ref{thm:LNS} on the maximal minimizer for the limit case $\k=h_\Om$, originally proved in~\cite[Theorem~1.4 and Remark~5.2]{LNS17}. There are two immediate consequences to this theorem. Firstly, we show in Corollary~\ref{cor:nestedness} the nestedness of minimizers for increasing values $\k_2>\k_1$, provided that $\Om$ has no necks of radii $\kappa_1^{-1}$ and $\kappa_2^{-1}$. Secondly, we show that Theorem~\ref{thm:LNS} can be ``improved'', in the following sense. In order to apply it, one needs to know a priori the value of the constant $h_\Om$, or at least to ensure that $\Om$ has no necks of radius $r$ for a range of values such that $h_\Om^{-1}$ falls within. If this happens, \emph{then} $r$ is the unique positive solution of $\pi \rho^2=|\Om^\rho|$. In Corollary~\ref{cor:improvement}, we prove that one can ``reverse'' these operations. By this, we mean that one can consider the unique positive solution $r$ to $\pi \rho^2=|\Om^\rho|$ and \emph{then} check if the set has no necks of radius $r$. If it does, then $r$ is the inverse of $h_\Om$ and the maximal Cheeger set is $\Om^r \oplus B_r$.\par

We mention that, thanks to the above result, one derives an extension of a result by Chen (see~\cite{Che80, Giu78}, or~\cite{FK02, KL06} for convex sets). Chen's theorem provides a criterion for a set $\Om$ to be the unique Cheeger set of itself. This also follows from a more general criterion related to self-minimizers of the prescribed curvature functional $\cF_\k$, to appear in the forthcoming paper~\cite{Sar19}.\par

Finally, notice that any nontrivial minimizer $E_\k$ of $\cF_\k$ is also a set attaining the minimum of the isoperimetric profile
\[
\mathcal{J}(V)= \inf \{\,P(F): F\subset \Om,|F|=V\,\},
\] 
relatively to the volume $V=|E_\k|$. Thanks to Theorem~\ref{thm:main} we are in a position to exhibit the minimizers of $\mathcal{J}(V)$ relatively to volumes $V\ge |E^m_{h_\Om}|$, provided that $\Om$ has no necks of radius $r$, for all $r \in (0,h_\Om^{-1}]$. Specifically, the following result holds.

\begin{thm}\label{thm:m_iso}
Let $\Om$ be a Jordan domain with $|\de \Om|=0$. Assume $\Om$ has no necks of radius $r=\k^{-1}$, for all $r \in (0, h_\Om^{-1}]$. Then, for all volumes $V\ge |E^m_{h_\Om}|$, there exists $\k \in  [h_\Om, +\infty)$ and a minimizer $E_\k$ of $\cF_\k$ such that 
\begin{align*}
|E_\k|=V, && \mathcal{J}(V) = P(E_\k).
\end{align*}
\end{thm}

Under the same hypotheses of Theorem~\ref{thm:m_iso}, we show the convexity of the isoperimetric profile $\mathcal{J}$ for $V\ge |E^m_{h_\Om}|$ by observing that it coincides with the Legendre transform of the convex function $\mathcal{G}\colon \k\mapsto -\min \cF_\k$, defined on $[h_\Om, +\infty)$, see Proposition~\ref{prop:legendre} and Corollary~\ref{cor:convexity}. This agrees with the results of~\cite{FLSS18} relatively to a relaxation of the isoperimetric profile. For the sake of completeness, we recall that the above theorem was known in the convex case, see~\cite[Theorem~3.32]{SZ97}. In the $n$-dimensional convex case, existence and uniqueness were discussed in~\cite[Section~4]{ACC05} (as well as in the Gaussian convex case~\cite[Theorem~23]{CMN10}).

\section{Properties of minimizers}\label{sec:properties}

Most of the proofs of the results presented in this section are not given, since they are easy adaptations from the limit case $\k=h_\Om$. The interested reader is referred to the original ones for which we give a precise reference.\par

We remark that throughout this section the no neck condition is \emph{never} enforced. Same goes for the request that $\Omega$ is a Jordan domain but for Section~\ref{ssec:2}. The results contained here apply generally to any minimizer in an open, bounded set $\Om \subset \R^2$.\par

First of all, notice that any minimizer $E_\k$ of $\cF_\k$ enjoys many regularity properties which come from the standard regularity theory of perimeter minimizers. Among these, the fact that $\de E_\k \cap \Om$ has constant (mean) curvature equal to $\k$, which is the reason why the functional is usually referred to as the \emph{prescribed (mean) curvature functional}. We collect these regularity properties of the boundary in the next proposition.

\begin{prop}\label{prop:properties}
Let $E_\k$ be a minimizer of $\cF_\k$ relatively to $\Om \subset \R^2$. Then, the following statements hold true:
\begin{itemize}
\item[(i)] $\de E_\k \cap \Om$ is analytic and coincides with a countable union of circular arcs of curvature $\k$, with endpoints belonging to $\de\Om$;
\item[(ii)] the length of any arc in $\de E_\k \cap \Om$ cannot exceed $\pi \k^{-1}$;
\item[(iii)] if $x\in \de E_\k \cap \de^* \Om$, then $x\in \de^*E_\k$ and $\nu_\Om(x) = \nu_{E_\k}(x)$.
\end{itemize}
\end{prop}

Point~(i) is nowadays standard, and one can refer to~\cite[Section~17.3]{Mag12book}. Point~(ii) can be proved as in~\cite[Lemma~2.11]{LP16}. Point~(iii) is well known for a Lipschitz $\Om$, see for instance~\cite{GMT81}; see also~\cite[Theorem~3.5]{LS18a} for a proof valid for every $\Om$ with finite perimeter.\par

We recall the notion of P-connectedness which in the theory of sets of finite perimeter replaces the usual notion of connectedness, and from now onwards whenever we write connected it is understood to be P-connected. Given a set $A$ of finite perimeter we say that it is \emph{decomposable} if there exists a partition $(E,F)$ of $A$ such that  $P(A)=P(E)+P(F)$ and both $|E|$ and $|F|$ are strictly positive. We say that it is \emph{indecomposable} if it is not decomposable. Given any set of finite perimeter $A$, there exists a unique finite or countable family $\{E_i\}_i$ of pairwise disjoint indecomposable sets with $|E_i|>0$ such that  $P(A)=\sum_i P(E_i)$, see~\cite[Theorem~1]{ACMM01}. We shall call each of these sets $E_i$ a \emph{P-connected component} of $A$.

\begin{figure}[t]%
\centering
	\subfigure[The maximal Cheeger set of the balanced dumbell.]{\includegraphics[trim={.5cm 4cm .5cm 5cm},clip, width=.33\textwidth]{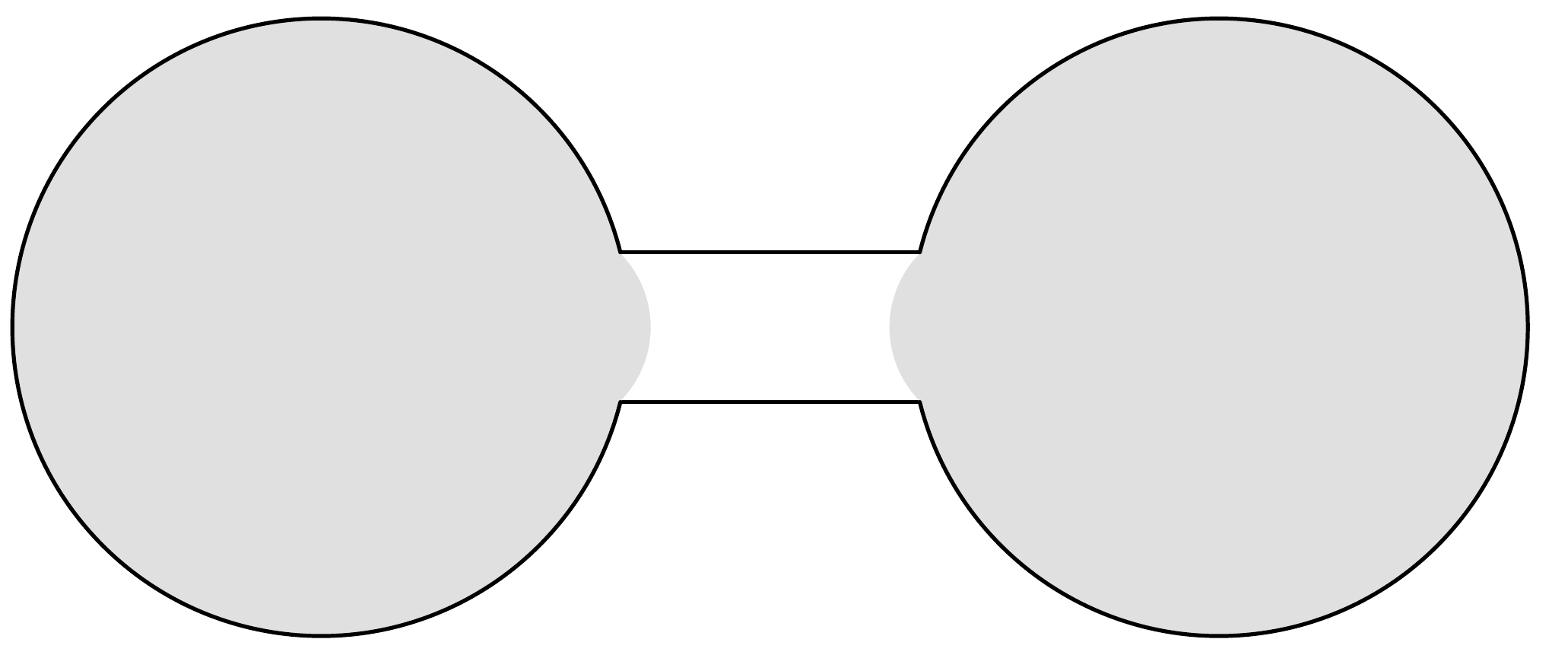}}\\%
	\subfigure[A minimal Cheeger set of the balanced dumbell. \label{fig:bal_barbell_c1}]{\includegraphics[trim={.5cm 4cm .5cm 5cm},clip, width=.33\textwidth]{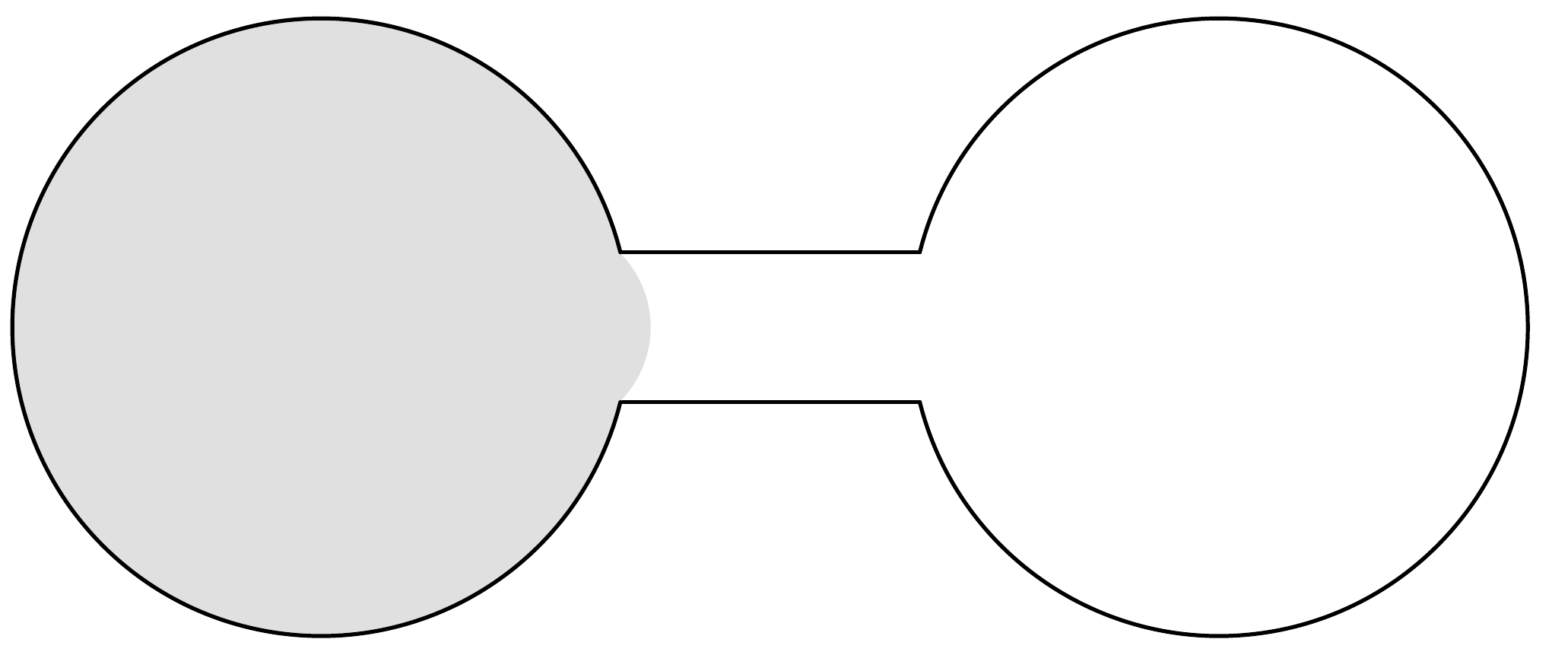}}%
	\hspace{2.5cm}
	\subfigure[A minimal Cheeger set of the balanced dumbell. \label{fig:bal_barbell_c2}]{\includegraphics[trim={.5cm 4cm .5cm 5cm},clip, width=.33\textwidth]{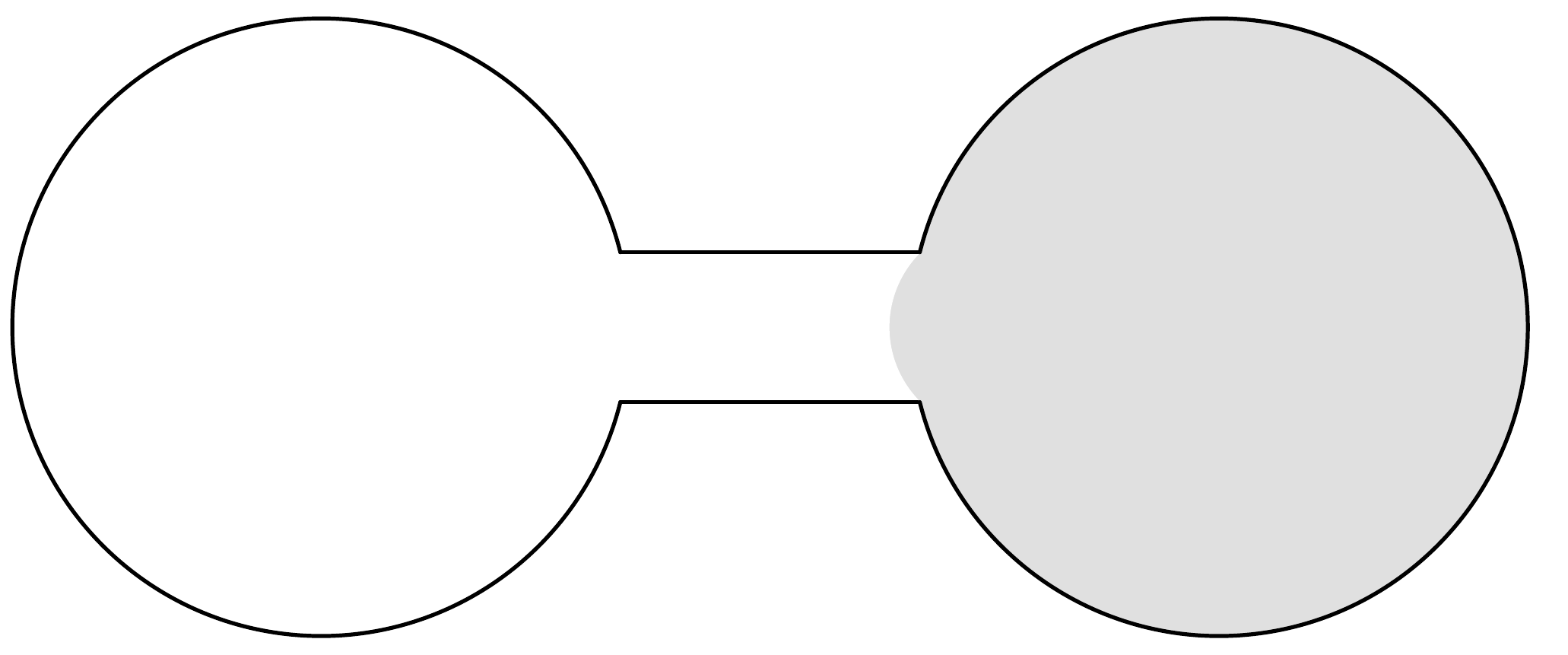}}%
\caption{The above figures show all the nontrivial minimizers of the prescribed curvature functional for $\k=h_\Om$, i.e.~the Cheeger sets of $\Om$, with $\Om$ a balanced dumbell.}
\label{fig:bal_barbell}
\end{figure}

In the next proposition we show that there exist both maximal and minimal minimizers of $\cF_\k$, which we recall we defined in Definition~\ref{def:min_max}.

\begin{prop}\label{prop:min_max}
There exists a unique maximal minimizer of $\cF_\k$, which is given by the union of all minimizers. There exist minimal minimizers of $\cF_\k$. Moreover, in the case $k>h_\Om$ one has the uniqueness of the minimal minimizer.
\end{prop}

\begin{proof}
We start noticing the following fact. If $E_\k$ and $F_\k$ are both minimizers, then $E_\k \cap F_\k$ and $E_\k \cup F_\k$ are minimizers as well, i.e.~the class of minimizers is closed under countable unions and intersections. Indeed, by the well-known inequality  (see for instance~\cite[Lemma~12.22]{Mag12book})
\[
P(E_\k \cup F_\k) + P(E_\k \cap F_\k) \le P(E_\k) +P(F_\k)\,,
\]
we have
\begin{align*}
P(E_\k&)+P(F_\k) - 2\min \cF_\k = \k|E_\k| +\k|F_\k| = \k|E_\k\cup F_\k|+ \k|E_\k\cap F_\k|\\
&\le P(E_\k \cup F_\k) + P(E_\k \cap F_\k)  - 2\min \cF_\k \le P(E_\k)+P(F_\k) - 2\min \cF_\k 
\end{align*}
thus all inequalities are equalities. Hence, we get
\begin{equation}\label{eq:ineq_per2}
P(E_\k \cap F_\k) - \k|E_\k \cap F_\k| = P(E_\k \cup F_\k) -\k|E_\k \cup F_\k|= \min \cF_\k\,.
\end{equation}
Let now $\{F^i_{\k}\}_i$ be a countable family of minimizers. Let $U_\k = \cup_i F^i_{\k}$ and $I_\k = \cap_i F^i_{\k}$. Then, thanks to~\eqref{eq:ineq_per2} and the lower semicontinuity of the perimeter, one readily shows that $U_\k$ and $I_\k$ are minimizers too. Notice that in the case $\k = h_\Om$, one can have $I_\k = \emptyset$, i.e.~the trivial minimizer. However, this can be excluded by requiring $\cap_{i\le j} E^i_\k \neq \emptyset$ for all $j$. Indeed, any nontrivial minimizer satisfies a uniform lower bound on the volume, see Proposition~\ref{prop:vol_bound}  below. Finally, observe that if two minimal minimizers $E_\k$ and $F_\k$ have a nonnegligible intersection, then the intersection is also a minimal minimizer and therefore $E_\k = F_\k$. This also shows that two distinct minimal minimizers must be P-connected components of their union. Hence, if we assume $\k>h_\Om$ we have $\cF_\k[E_\k] <0$ for every minimal minimizer $E_\k$. Thus, the existence of another minimal minimizer $F_\k\neq E_\k$ would lead to $\cF_\k[E_\k\cup F_\k] = \cF_\k[E_\k] + \cF_\k[F_\k] < \cF_\k[E_\k]$, against minimality. This shows that when $\kappa>h_\Om$ the minimal minimizer is unique.
\end{proof}

\begin{rem}\label{rem:uniqueness_minimal}
It is rather interesting to notice that there exists a unique, nontrivial minimal minimizer whenever $\k>h_\Om$, given precisely by the intersection of all minimizers. This is in contrast with the limit case $\k = h_\Om$, where one can have multiple minimal minimizers, as the dumbell in Figure~\ref{fig:bal_barbell} shows. The reason is that, for $\k=h_\Om$, any connected component of a minimizer is a minimizer itself (see Figures~\ref{fig:bal_barbell_c1} and~\ref{fig:bal_barbell_c2}), while this is false for $\k>h_\Om$ (for comparison, see Figure~\ref{fig:bal_barbell_k}).
\end{rem}

\begin{figure}[t]
\centering
\includegraphics[trim={.5cm 4cm .5cm 5cm},clip, width=.75\textwidth]{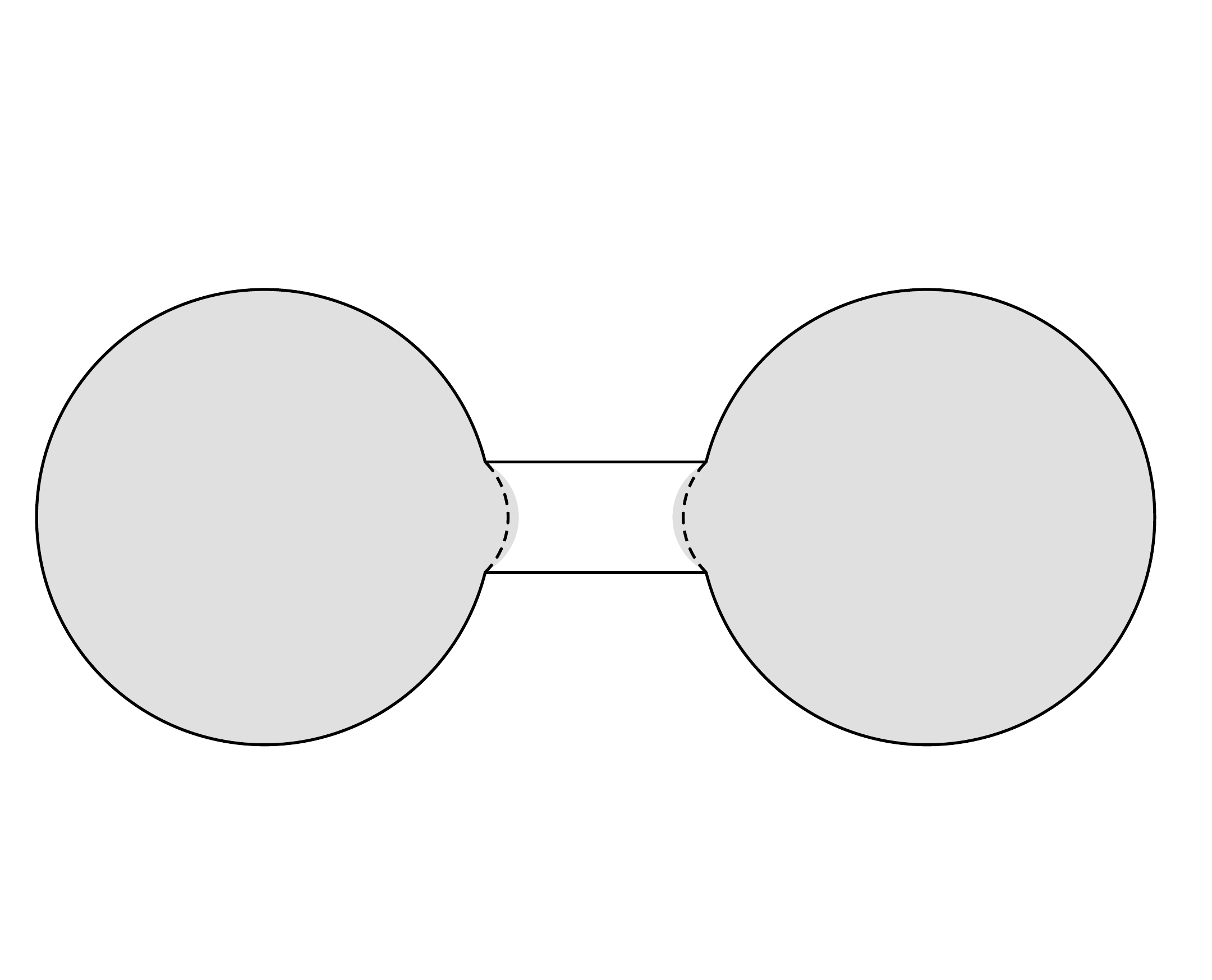} 
\caption{The shaded area represents the minimizer of the prescribed curvature functional for $\k$ close to $h_\Om$, while the dashed curves are the interior boundary of the maximal Cheeger set. Each of the connected components $E^i_\k$ is such that $\cF_\k[E^i_\k] < 0$, hence a component alone is not a minimizer.}
\label{fig:bal_barbell_k}
\end{figure}

The following lower bound to the volume of any connected component of a minimizer is readily established.

\begin{prop}\label{prop:vol_bound}
Let $E_\k$ be a minimizer of $\cF_\k$. Then, any of its connected components $E_\k^i$ has volume bounded from below by $4\pi \k^{-2}$.
\end{prop}

\begin{proof}
If $\k=h_\Om$ this is straightforward from the isoperimetric inequality and the well-known fact that any connected component of a Cheeger set is a Cheeger set itself. Suppose now that $\k > h_\Om$ and without loss of generality that $E_\k$ is decomposable, i.e.~there exist $E^1_\k, E^2_\k \subset E_\k$ with $|E^1_\k|\cdot|E^2_\k|>0$ and such that 
\begin{align*}
|E_\k|= |E^1_\k| + |E^2_\k|, && P(E_\k) = P(E^1_\k) + P(E^2_\k).
\end{align*}
Assume by contradiction that $|E_\k^1| < 4\pi\k^{-2}$ and denote by $B_{E_\k^1}$ the ball with same volume of $E^1_\k$. Its radius $r_{E^1_\k}$ is strictly less than $2\k^{-1}$. Thus,
\begin{align*}
\cF_\k[E^1_\k] &= P(E^1_\k)-\k|E^1_\k| \ge P(B_{E^1_\k}) - \k|B_{E^1_\k}|\\
&= 2\pi r_{E^1_\k}  - \k \pi r_{E^1_\k}^2 = \pi r_{E^1_\k}(2-\k r_{E^1_\k}) > 0.
\end{align*}
Therefore $\cF_\k [E^2_\k] < \cF_\k [E_\k]$, against the minimality of $E_\k$.
\end{proof}

Finally, we recall the \emph{rolling ball lemma}~\cite[Lemma~2.12]{LP16}, which was later refined~\cite[Lemma~1.7]{LNS17}. This still holds for general $\k$, and the proof is a straightforward adaptation of the original lemma.

\begin{lem}[Rolling ball]\label{lem:rollingball}
Let $\k\ge h_\Om$ be fixed, and let $E^M_\k$ be the maximal minimizer of $\cF_\k$. If $E^M_\k$ contains a ball $B_r(x_0)$ of radius $r=\k^{-1}$, then it contains all balls of same radius that can be reached by rolling $B_r(x_0)$, i.e.~it contains any ball $B_r(x_1)$ such that there exists a continuous curve $\g\colon[0,1]\to \Omega$ with $\g(0)=x_0$, $\g(1)=x_1$ and $B_r(\g(t))\subset \Om$ for all $t\in[0,1]$.
\end{lem}

\subsection{Additional properties when \texorpdfstring{$\Om$}{Om} is a Jordan domain}\label{ssec:2}

Here, we state a few additional properties of minimizers when $\Om$ is a Jordan domain. Their proofs are omitted as they closely follow the corresponding ones presented in~\cite{LNS17} for the case $\k=h_\Om$. 

\begin{prop}\label{prop:palla_interna}
Suppose $\Om\subset \R^2$ is a Jordan domain with $|\de \Om| = 0$, and let $E_\k$ be a minimizer of $\cF_\k$. Then,
\begin{itemize}
\item[(i)] the curvature of $\de E_\k$ is bounded from above by $\k$ in both variational and viscous senses;
\item[(ii)] $E_\k$ is Lebesgue-equivalent to a finite union of simply connected open sets, hence its measure-theoretic boundary $\de E_\k$ is a finite union of pairwise disjoint Jordan curves;
\item[(iii)] $E_\k$ contains a ball of radius $\k^{-1}$.
\end{itemize}
\end{prop}
The definitions of curvature in variational and in viscous senses, notions that appear in the above proposition, can be found resp.~in~\cite{BGM03} and~\cite[Definition~2.3]{LNS17}. The proof of~(i) is obtained by mimicking~\cite[Lemma~2.2 and Lemma~2.4]{LNS17}. The proof of~(ii) follows by arguing as in~\cite[Propositions~2.9 and~2.10]{LNS17}. The proof of claim (iii) follows from (i) and (ii) combined with~\cite[Theorem~1.6]{LNS17}.

\section{The set difference \texorpdfstring{$\Om^r \setminus \overline{\interior(\Om^r)}$}{Omr - clos(int(Omr))}}\label{sec:prelim}

Here we prove Proposition~\ref{prop:struttura_diff}, i.e.~the structure of the set difference $\Om^r \setminus \overline{\interior(\Om^r)}$, under the assumption that $\Om$ has no necks of radius $r$. According to Definition~\ref{def:no_necks}, this means that given any two balls $B_r(x_0)$ and $B_r(x_1)$ contained in $\Om$, there exists a continuous curve $\g \colon [0,1] \to \Om^r$ such that $\g(0)=x_0$ and $\g(1)=x_1$. Thanks to~\cite[Theorem~1.8]{LNS17}, we can further assume $\g$ to be of class $\textrm{C}^{1,1}$ with curvature bounded by $1/r$.\par

We now lay down some notation we shall use throughout the paper from now onwards. Given a regular curve $\g\colon [0,1] \to \R^2$ of class $\textrm{C}^{1,1}$, we set 
\begin{align*}
\g^\prime(0) = \lim_{t\to 0^+}\g^\prime(t)\,, && \g^\prime(1) = \lim_{t\to 1^-}\g^\prime(t).
\end{align*}
We denote by $\nu(t)$ the renormalization of $\g^\prime(t)$, i.e.
\[
\nu(t)= \frac{\g^\prime(t)}{|\g^\prime(t)|}.
\]
Owing to the regularity of $\g$, $\nu(t)$ is continuous and defined on the whole interval $[0,1]$. Given $r>0$, we define the open half-ball 
\begin{equation}\label{eq:semipalla}
B_r^+(\g(t))=\{\,z\in \R^2: |z-\g(t)|<r, (z-\g(t))\cdot \nu(t) >0 \,\},
\end{equation}
and the relatively open half-circle 
\[
S_r^+(\g(t))=\{\,z\in \R^2: |z-\g(t)|=r, (z-\g(t))\cdot \nu(t) >0 \,\},
\]
that are ``oriented in the direction $\nu(t)$'' (note that we have dropped the explicit dependence on $\nu(t)$ in the notation). Finally, the endpoints of $S_r^+(\g(t))$ are denoted by
\begin{align}\label{eq:ztpm}
z^+_t = \g(t)+r\nu(t)^\perp, && z^-_t = \g(t)-r\nu(t)^\perp.
\end{align}

\begin{proof}[Proof of Proposition~\ref{prop:struttura_diff}]
If $\Om^r$ and $\overline{\interior(\Om^r)}$ agree, there is nothing to prove. Let us suppose then that there exists $x\in \Om^r\setminus \overline{\interior(\Om^r)}$, and let us denote by $\Pi_x$ its ``projection set'', i.e.
\[
\Pi_x=\{y\in \de \Om\,:\, |y-x|=r\}\,.
\]
We split the proof in two steps, following points~(a) and~(b) of the statement.\par
\emph{\underline{(a) The case $\interior(\Om^r)=\emptyset$.}} We can distinguish three subcases, according to the properties of the projection set $\Pi_x$.\par
\begin{itemize}
\item[(a1)] For all directions $\nu \in \S^1$, there exist two points $y_1, y_2 \in\Pi_x$ such that
\[
\nu \cdot (y_1-x) < 0 < \nu \cdot (y_2-x).
\]
\item[(a2)] There exist a direction $\nu \in \S^1$ and two distinct points $y_1\,, y_2 \in \Pi_x$ such that
\begin{align*}
\nu \cdot (y_1-x) = \nu \cdot (y_2-x) = 0, && \nu \cdot (y-x) \ge 0,\quad \forall y\in \Pi_x.
\end{align*}
\item[(a3)] There exist a direction $\nu \in \S^1$ and $\delta>0$ such that
\[
\nu \cdot (y-x) \ge \delta, \quad \forall y\in \Pi_x.
\]
\end{itemize}
We start noticing that case~(a3) can never happen. Indeed, one could easily show that $x-\varepsilon \nu \in \interior(\Om^r)$, for $\varepsilon$ sufficiently small which contradicts $\interior(\Om^r)=\emptyset$.
\par
In case~(a1), it is immediate to see that $x$ is an isolated point in $\Om^r$. Then, as $\Om$ has no necks of radius $r$ we infer that $\Om^r =\{x\}$, i.e.~it is a constant curve. We are then left with case~(a2), which implies that $\Om^r$ satisfies a bilateral ball condition of radius $r$ at $x$, which means there exist two balls of radius $r$, $B_1$ and $B_2$, such that $\overline{B_1}\cap \overline{B_2} = \{x\}$, and locally $\Om^r \cap (B_1\cup B_2)=\emptyset$. As this holds for any choice of $x$, and as $\Om^r$ is path-connected, this necessarily means that $\Om^r=\g$, with $\g$ a $\textrm{C}^{1,1}$ curve with curvature bounded by $r^{-1}$. Further, this curve cannot be a loop:~since $\Om$ is a Jordan domain, this would imply that $\interior(\Om^r)\neq \emptyset$. Therefore, $\Om^r$ is diffeomorphic to the closed segment $[0,1]$.\par

\emph{\underline{(b) The case $\interior(\Om^r)\neq \emptyset$.}} 
 We fix $y\in\interior(\Om^r)$, which exists since by hypothesis this set is not empty. The assumption of no necks of radius $r$ paired with~\cite[Theorem~1.8]{LNS17} yields the existence of a $\textrm{C}^{1,1}$ curve $\g$, with curvature bounded by $r^{-1}$, such that  $\g(0)=x$ and $\g(1)=y$. Let $T>0$ be the first time for which $\g(T)\in \overline{\interior(\Om^r)}$.  Thanks to Zorn's lemma, we can extend $\g_{|[0,T)}$ to a maximal curve $\widetilde \g$  in $\Om^r \setminus \overline{\interior(\Om^r)}$. Moreover, by continuity we can extend $\widetilde \g$ to a closed interval, and up to a reparametrization we can assume it to be $[0,1]$. Without loss of generality, suppose that $\widetilde \g(0)=\g(T) \in \de(\interior(\Om^r))$. Hence, there are two possible cases:~either $\widetilde  \g(1)$ belongs as well to $\de(\interior(\Om^r))$; or $\widetilde \g(1)$ belongs to $\Om^r \setminus \overline{\interior(\Om^r)}$.\par
 By reasoning as in the first step, we notice that $x$ has at least two 2 antipodal projections in $\Pi_x$. By the bilateral ball condition, which holds at any $x\in \widetilde \g$, one can show that there exists $\e = \e_x>0$ such that
 \[
 \left( \Om^r \setminus \overline{\interior(\Om^r)} \right) \cap B_\e(x) = \widetilde \g \cap B_\e(x).
 \]
 We define $\G^1_r$ as the collection of connected components of $ \Om^r \setminus \overline{\interior(\Om^r)}$ that are diffeomorphic to the half-closed interval $(0,1]$, and similarly $\G^2_r$ as the collection of connected components that are diffeomorphic to the open interval $(0,1)$.
\par
We are left with showing that $\#\G^1_r<\infty$. Let us fix any $\g\in \G^1_r$ and let $x_\g=\g(1)$. By reasoning as in the first part of the proof, we have that $z_1^\pm$ as defined in~\eqref{eq:ztpm} belong to $\Pi_{x_\g}$. We claim that all $z\in B^+_r(x_\g)$ have as unique projection on $\Om^r$ the point $x_\g$, where $B^+_r(x_\g)$ is defined in~\eqref{eq:semipalla}. This proves that from any curve $\g \in \G^1_r$ stems a contribute to the volume of at least $\frac \pi 2 r^2$ . The finiteness of $|\Om|$ implies then the finiteness of the family $\G^1_r$.\par

To show this we argue by contradiction. Let us suppose that some $z\in B^+_r(x_\g)$ has as unique projection $y\in \Om^r$ with $y\neq x_\g$. By the no neck assumption there is a $\textrm{C}^{1,1}$ curve $\sigma$ from $x_\g$ to $y$, which lies in $\Om^r$. We claim that the loop constructed by concatenating $\sigma$, the segment $[x_\g,z]$ and the segment $[z,y]$ contains either $z_1^+$ or $z_1^-$ giving a contradiction to the simple connectedness of $\Om$.\par
This follows by noticing that both the segment $[z,y]$ and the curve $\sigma$ cannot pass across the segment $[z_1^-,z_1^+]$. First, assume by contradiction that $[z,y]$ crosses the open segment $[z_1^-,z_1^+]$ in $w$. Trivially, $w$ cannot coincide with $x_\g$ otherwise this contradicts $y$ being the closest point in $\Om^r$ to $z$. Moreover, as $w$ is in the open segment $[z_1^-, z_1^+]$ it projects uniquely on $x_\g$, therefore $|y-w|>|w-x_\g|$. By triangular inequality it immediately follows that $|x_\g-z|< |y-z|$ which is a contradiction.\par

Second, on the one hand $\sigma$  cannot pass through the points lying in the open segments $[x_\g,z_1^+]$ and $[x_\g,z_1^-]$ as all these have distance from the boundary less than $r$ (since $z_1^-, z_1^+ \in \Pi_{x_\g}$). On the other hand, for some $\e=\e(x_\g)<<1$ we have $\Om_r \cap B_\e(x_\g)=\g$. As $\g \in \G^1_r$ and $x_\g=\g(1)$ one has that $\Om^r\cap B_\e^+(x_\g)=\emptyset$. Thus, $\sigma \cap B^+_\e(x_\g)= \emptyset$. This establishes that all $z\in B_r^+(x_\g)$ have as unique projection on $\Om^r$ the point $x_\g$.
\end{proof}

\begin{rem}
As can be seen from the proof of the above proposition, we remark that any point $x$ belonging to $\g$, $x=\g(t)$, with $\g$ in either $\G^1_r$ or $\G^2_r$, has two projections on $\de \Om$ that are antipodal, given by $z^\pm_t$, defined in~\eqref{eq:ztpm}. This in particular implies that any strip
\begin{equation}\label{def:strip}
\mathcal{S}(\g) =\left\{\,\g(t) \pm \rho \nu(t)^\perp: t\in [0,1], \rho\in [0,r)\,\right\}
\end{equation}
is diffeomorphic to the rectangle $[0,1]\times (-r,r)$. Moreover, given any two curves $\g_1, \g_2 \in \G^1_r$ the strips $\mathcal{S}(\g_1), \mathcal{S}(\g_2)$ are pairwise disjoint. Finally, notice that the ``lateral boundart'' of $\mathcal{S}(\g)$
\begin{equation}\label{def:lateral_surface}
\de_L \mathcal{S}(\g) =\left \{\, z^\pm_t = \g(t) \pm r \nu(t)^\perp: t\in [0,1]\,\right \}
\end{equation}
is contained in $\de \Om$.
\end{rem}

\begin{rem}\label{rem:nogammadue}
Notice the following: if $\Om$ is a Jordan domain with no necks of radius $r$ for all $r\le R \le \inr(\Om)$, then for every $r<R$ the set $\Gamma^2_r$ is empty. Argue by contradiction and suppose $\exists \g \in \G^2_r$. The points $\g(0)$ and $\g(1)$ belong to $\de (\interior{(\Om^r)})$. Therefore, we can find a point $z_0 \in \interior{(\Om^r)}$ (resp. $z_1$) arbitrarily close to $\g(0)$ (resp. $\g(1)$). Clearly one has $r<\bar r$ where $\bar r= \min\{\,\dist(z_0; \de \Om); \dist(z_1; \de \Om) \,\}$ and without loss of generality we can suppose $\bar r<R$. As $\Om$ has no necks of radius $\bar r$, there exists a curve $\s$ joining these two points contained in $\Om^{\bar r}$. Being $\bar r > r$, the curves $\g$ and $\s$ cannot meet but in the endpoints. Therefore, by concatenating these two curves, and the segments $[\g(i), z_i]$ for $i=0,1$, one reaches a contradiction as in the proof of Proposition~\ref{prop:struttura_diff}.
\end{rem}

\section{Structure of minimizers}\label{sec:main}

In this section we give the proof of Theorem~\ref{thm:main}. The part concerning the structure of the maximal minimizer closely follows the one of~\cite[Theorem~1.4]{LNS17} for the case $\k=h_\Om$, while the one about the minimal minimizer relies on Proposition~\ref{prop:struttura_diff}. \par

We first need to prove that for $\k > h_\Om$, Proposition~\ref{prop:struttura_diff} applies, i.e.~that $\Om^{r}$ with $r=\k^{-1}$ has nonempty interior. We do so in the next lemma.

\begin{lem}\label{lem:struttura_diff}
Let $\Om$ be a Jordan domain and let $\k \ge h_\Om$. Assume that $\Om$ has no necks of radius $r=\k^{-1}$, then $\interior(\Om^r)$ is not empty.
\end{lem}

\begin{proof}
 Take any $\k>h_\Om$, and let $E_{h_\Om}$ be a Cheeger set of $\Om$. By Proposition~\ref{prop:palla_interna}~(iii) there exists a ball $B$ of radius $1/h_\Om$ such that $B\subset E_{h_\Om}\subset \Om$. Hence, for all $\k>h_\Om$, one has that $\Om^{1/\k}$ contains at least a ball of radius $1/h_\Om - 1/\k$. \par
We now settle the case $\k = h_\Om$. By the first part we already know that $\Om^r\neq \emptyset$, because $E_{h_\Om}$ contains at least a ball of radius $r=h_\Om^{-1}$. Argue by contradiction and suppose that $\interior(\Om^r) =\emptyset$, i.e.~$\Om^{1/h_\Om}$ is, by Proposition~\ref{prop:struttura_diff}~(a), a (possibly constant) $\textrm{C}^{1,1}$ curve homeomorphic to a closed segment, thus $|\Om^{1/h_\Om}|=0$. This contradicts the inner Cheeger formula~\eqref{eq:inner_Cheeger_formula} which states $|\Om^{1/h_\Om}|=\pi h^{-2}_\Om$.
\end{proof}

\begin{rem}
Notice that in the proof of the above lemma we use that $\Om$ has no necks of radius $\k^{-1}$ only in the case $\k=h_\Om$. We believe that this is not necessary but we do not have an immediate proof of this fact. In any case, the assumption that $\Om$ is a Jordan domain cannot be avoided:~one needs it to apply Proposition~\ref{prop:palla_interna}~(iii). Moreover, in the case $\k=h_\Om$ the claim surely fails without such a hypothesis:~a counterexample is given by annuli, or more generally by curved annuli~\cite{KP11}.
\end{rem}

\begin{lem}\label{lem:reach}
Let $\Om$ be a Jordan domain with no necks of radius $r$, and assume that $\interior(\Om^r)\neq \emptyset$. Then, the compact sets
\[
C_t = \overline{\interior(\Om^r)} \cup \bigcup_{\g \in \Gamma^2_r} \g \cup \bigcup_{\g \in \G^1_r} \gamma([0,t])\,,\qquad t\in[0,1]\,,
\]
are such that $\reach(C_t)\ge r$. Moreover, they are simply connected.
\end{lem}

For the sake of completeness we recall the definition of reach for a closed set $A$, which was introduced in the seminal paper~\cite{FedererCM}. The reach of a closed set $A$ is
\[
\reach(A) = \sup \{ r : \forall x \in A\oplus B_r\,, \, x \text{ has a unique projection onto } A\}.
\]

\begin{proof}
Notice that for $t=1$ the claim corresponds to~\cite[Lemma~5.1]{LNS17}. To prove the claim we need to show that all the points in
\begin{equation}\label{eq:def_A_t}
A_t = \left(\overline{\interior(\Om^r)} \cup \bigcup_{\g \in \Gamma^2_r} \g \cup \bigcup_{\g \in \G^1_r} \gamma([0,t])  \right) \oplus B_r\,,
\end{equation}
have a unique projection on $C_t$, for all $t\in[0,1]$. Let $\bar t>0$ and $\g\in\G^1_r$ be fixed. First, consider any point $x$ in the strip $\mathcal{S}(\g_{|(0,\bar t)})$ defined as in~\eqref{def:strip}. We can split its boundary as $\de_L \mathcal{S}(\g_{|(0,\bar t)})) \cup [z_0^+,z_0^-] \cup [z_{\bar t}^+, z_{\bar t}^-]$, where $\de_L \mathcal{S}(\g_{|(0,\bar t)})$ is defined as in~\eqref{def:lateral_surface}, $z_t^\pm$ as in~\eqref{eq:ztpm}, and $[p,q]$ denotes the segment with endpoints $p$ and $q$. Argue by contradiction and suppose that $x$ has not a unique projection on $C_{\bar t}$. As it has unique projection on $C_1$, say $z$, one has $z=\widetilde \g(\tau)$ for some $\widetilde \g \in \G^1_r$ and $\tau > \bar t$.  Clearly all points on the segment $[x,z]$ project on $C_1$ onto $z$. If we show that this segment cannot cross $\de \mathcal{S}(\gamma)$ we get a contradiction. Trivially, the segment cannot cross the lateral boundary $\de_L \mathcal{S}(\g)$, as this is a subset of $\de E^M_\k\subset \de \Om$ and those points have distance $r$ from $C_1$.  Furthermore, since the balls $B_r(z_t^\pm)$ with $t=0,\bar t$ are disjoint from $C_1$, the segment $[x,z]$ cannot cross $[z_0^+,z_0^-]$ but in $\g(0)$ (equivalently, $[z_{\bar t}^+,z_{\bar t}^-]$ but in $\g(\bar t)$) which gives a contradiction. We remark as well that the points $\g(t)+\rho \nu(t)^\perp$ with $\rho < r$ project uniquely onto $\g(t)$, thanks to well-known properties of the strip (see~\cite{LP16}).\par

Second, consider $x$ in the open half-ball $B^+_r(\g(\bar t))$ defined in~\eqref{eq:semipalla}. Arguing as in the last part of the proof of Proposition~\ref{prop:struttura_diff} we find that $x$ projects uniquely on $\g(\bar t)$. Indeed, suppose that $x\in B^+_r(\gamma(\bar t))$ has a projection $z\in C_{\bar t}$ with $z\neq \g(\bar t)$. As $\Om$ has no necks of radius $r$ we find a simple curve $\sigma$ that runs from $\g(\bar t)$ to $z$. We claim that the loop obtained by concatenating $\sigma$ and the segments $[\g(\bar t), x]$, $[x, z]$ contains either $z^+_{\bar t}$ or $z^-_{\bar t}$ against the simple connectedness of $\Om$. This follows again by noticing that $\sigma$ cannot go across the open segments $[z^\pm_{\bar t}, \g(\bar t)]$ and cannot intersect any point on $\g((\bar t, 1])$ since none of these can be connected to $z$ without passing through $\g(\bar t)$.\par
Third, we are left with showing that the points in
\[
D=  A_0 \setminus \bigcup_{\gamma\in \G^1_r} B^+_r(\gamma(0)),
\]
have unique projection on $C_t$, for all $t$. Notice that  for all $\g \in \G^1_r$ the set $D$ is pairwise disjoint with:~(i)~the strip $\mathcal{S}(\g)$ defined in~\eqref{def:strip};~(ii)~the open half-ball $B^+_r(\g(1))$. Therefore, any $x\in D$ cannot be of the form
\begin{align}\label{eq:form_of_x}
  \begin{aligned}
   &\g(t)\pm \rho \nu(t)^\perp, \\       &\g(1)\pm \rho \nu,
  \end{aligned}
  &&
  \begin{aligned}
  &t\in (0,1),\rho \in [0,r), \\       &\nu: \nu \cdot \nu(1)\ge 0, \rho \in [0,r).
  \end{aligned}
\end{align}
Fix a point $x\in D$. As $D\subset C_1\oplus B_r$, $x$ has a unique projection $z$ on $(C_1\oplus B_r)^r=C_1$, and $\dist(x; C_1) < r$. We claim that $z\in C_0\subset C_1$, which would imply the uniqueness of the projection of $x$ on $C_t$, for all $t$. This is equivalent to say that $z\notin \bigcup_{\G^1_r}\g((0,1])$. By contradiction suppose that $z=\g(t)$ for some $\g\in \G^1_r$ and $t\in (0,1]$. Necessarily all points on the segment $[x, \g(t)]$ project on $\g(t)$. If $t<1$, this implies that the segment has direction $\nu(t)^\perp$ by orthogonality. If $t=1$, this implies that the segment has direction $\nu$ such that $\nu \cdot \nu(1) \ge 0$. Hence, as $\dist(x; C_1) < r$, $x$ is of the form given in~\eqref{eq:form_of_x}, against the assumption.\par
We are left with showing that $C_t$ is simply connected for all $t\in[0,1]$. As $\Om$ has no necks of radius $r$ and by the definition of $C_t$, we infer the path-connectedness of $C_t$. The simple connectedness is then a straightforward consequence of $\Om$ being a Jordan domain, thus simply connected.
\end{proof}

We are now ready to prove our main theorem.

\begin{proof}[Proof of Theorem~\ref{thm:main}]
As the proof of the structure of the maximal minimizer is substantially the same of the case $\k=h_\Om$ detailed in~\cite[Theorem~1.4]{LNS17}, we here only sketch it. Let $E^M_\k$ be the maximal minimizer. By Proposition~\ref{prop:palla_interna}~(iii) $E^M_\k$ contains a ball of radius $r=\k^{-1}$. The assumption of no necks of radius $r$ coupled with Lemma~\ref{lem:rollingball} gives the inclusion $E^M_\k \supseteq \Om^r \oplus B_r$.\par
To show the opposite inclusion one argues by contradiction. Yet, this part is much more technical and requires using tools such as the structure of the \emph{cut-locus} and the characterization of \emph{focal points}. Since these play no role in this article besides this part of the proof, we do not comment further and we simply refer the interested reader to the original proof for $\k=h_\Om$ available in~\cite[Theorem~1.4]{LNS17} which can be followed step by step.\par

Let us now discuss the structure of the minimal minimizer. We split the proof in three steps. By Lemma~\ref{lem:struttura_diff} and Proposition~\ref{prop:struttura_diff}, we know that $\Om^r \setminus \overline{\interior(\Om^r)}$ consists of the two (possibly empty) families $\G^1_r$ and $\G^2_r$ satisfying properties~(i)--(v) of Proposition~\ref{prop:struttura_diff}~(b). According to the notation introduced in Lemma~\ref{lem:reach} we denote by $A_t$ the set defined in~\eqref{eq:def_A_t}. Hence, we aim to prove that $A_0$ is the unique minimal minimizer of $\cF_\kappa$.\par
\emph{\underline{Step (i).}} For each $t\in[0,1]$ the set $A_t$ is a minimizer. Notice that for $t=1$, $A_1=E^M_\kappa$, thus the minimality is trivially true. According again to the notation and to the statement of Lemma~\ref{lem:reach}, $A_t=C_t\oplus B_r$ and the set $C_t$ is such that $\reach(C_t)\ge r$ and it is simply connected. Hence, by Steiner's formulas (see~\cite{FedererCM} and~\cite[Section~2.3]{LNS17}) we have
\begin{align*}
|A_t| = |C_t| + r\mathcal{M}_o(C_t) + \pi r^2\,, && P(A_t) = \mathcal{M}_o(C_t) + 2\pi r\,,
\end{align*}
where $\mathcal{M}_o(F)$ is the \emph{outer Minkowski content of $F$}, i.e.
\[
\mathcal{M}_o(F) = \lim_{r\to 0} \frac{|F\oplus B_r|-|F|}{r}\,.
\]
As $r=\kappa^{-1}$ and $|C_t|=|C_1|$ for all $t\in[0,1]$ it is immediate to check that $\cF_\k[A_t]=\cF_\k[E^M_\k]$, for all $t\in[0,1]$ which yields the claim.\par

\emph{\underline{Step (ii).}} Let $\k>h_\Om$. By Proposition~\ref{prop:min_max} the minimal minimizer is unique, thus we necessarily have  $A_0\supseteq E^m_\k$. Let us suppose that the inclusion is strict, and let $p\in A_0\setminus \overline{E^m_\k}$. By definition of $A_0$ we find $y$ either in $\overline{\interior(\Om^r)}$ or in $\g$ for some $\g \in \Gamma^2_r$, such that $p\in B_r(y)$. By Proposition~\ref{prop:palla_interna}~(iii), we find $z\in E^m_\k$ such that $B_r(z)\subset E^m_\k$. By the assumption of no necks of radius $r$, there is a $\textrm{C}^{1,1}$ curve $\sigma$ contained in $\Om^r$ such that  $\sigma(0)=z$ and $\sigma(1)=y$.\par
Let us denote by $t^*$ the last time for which $B_r(\sigma(t))\subset E^m_\k$ for all $t\le t^*$, which by hypothesis satisfies $t^*<1$. We claim that $\de E^m_\k \cap \Om$ contains the half-circle $S^+_r(\s(t^*))$ of length $\pi r$. To show this, let us fix $\e>0$ and take $t\in(t^*,1)$ sufficiently close to $t^*$, such that $B_r(\s(t))$ is not contained in $E^m_\k$. Therefore, the set $B_r(\s(t))\cap \de E^m_\k$ is nonempty, hence we can select $x_t\in B_r(\s(t))\cap \de E^m_\k$ minimizing the distance from $\s(t)$. Let $S_t$ be the connected component of $\de E^m_\k\cap \Om$ containing $x_t$, which is actually an arc of circle of radius $r$ with endpoints on $\de \Om$. Since the endpoints of $S_t$ lie outside $B_r(\s(t))\cup B_r(\s(t^*))$, and since the boundaries of these two balls have the same curvature as $S_t$, we conclude that the length of $S_t$ must be at least $\pi r -\e$, provided $t$ and $t^*$ are close enough.
Since $\de E^m_\k	\cap \Om$ has finitely many components of length greater than or equal to $\pi r -\e$, we find a sequence $t_n$ converging to $t^*$ such that $S=S_{t_n}$ is constant and intersects every ball $B_r(\s(t_n))$. Since $t_n$ converges to $t^*$, the distance of $S$ from $\de B_r(\s(t^*))$ is smaller than any positive constant, hence we conclude that $S$ is a half-circle contained in $\de B_r(\s(t^*))$. By construction, we get as well that $S=S_r^+(\s(t^*))$.\par

Consequently, we have $\s(t^*)= \g(\tau^*)$ for some $\g\in \G^2_r$ and $\tau^*\in [0,1]$. It is not restrictive to assume that $\dot{\g}(\tau^*) = \lambda\dot{\s}(t^*)$ for some $\lambda >0$. Pick a point $w$ arbitrarily close to $\g(1)$ in the connected component of $\interior (\Om^r)$ whose boundary contains $\g(1)$, and let $\widetilde \g:[0,1]\to \R^2$ be a $\textrm{C}^{1,1}$ curve contained in $\Om^r$ and connecting $\s(t^*)$ to $w$ (its existence is granted by the no necks assumption). We are now ready to define a one-parameter family of minimizers, by rolling balls along $\widetilde \g$, that is, by applying the same construction as in the proof of Lemma~\ref{lem:rollingball}), and by exploiting the fact that $S_r^*= S_r^+(\widetilde \g(0)) = S_r^+(\s(t^*))$ is contained in $\de E^m_\k \cap \Om$. \par
Let us define for $t\in [0,1]$ the sets
\begin{align*}
E_t = \bigcup_{s\in[0, t]} S_r^+(\widetilde \g(s)),
\end{align*} 
and 
\[
D_t= E^m_\k \cup  E_t. 
\]
If $\overline{E_t}$ and $\overline{E^m_\k}\setminus \overline{S_r^*}$ are disjoint, one has (see~\cite[Section~3]{LP16})
\begin{align*}
P(D_t) - P(E^m_\k) = 2\ell_t, && |D_t \setminus E^m_\k| = 2r\ell_t,
\end{align*}
where $\ell_t$ is the length of the curve $\widetilde \g$ restricted to $(0,t)$. Then, from the above formulas, $D_t$ is a minimizer as well. Let $\tilde{t}$ be the supremum of $t\in [0,1]$ such that $\overline{E_t}\cap \overline{E^m_\k}\setminus \overline{S_r^*}=\emptyset$. By the lower semicontinuity of the perimeter, the set $D_{\tilde{t}}$ is still a minimizer. If $\tilde{t}=1$, a contradiction follows immediately. Indeed $D_1$ would be a minimizer such that $S^+_r(\widetilde \g(1))\subset (\de D_{1}\cap \Om)$ and, at the same time, $\de B_r(\widetilde \g(1))\cap \de \Om = \emptyset$, which would contradict Proposition~\ref{prop:properties}(i). If $0\le\tilde{t}<1$, there exists another connected component $\Sigma$ of $\de E^m_\k\cap \Om$, such that its closure tangentially meets the closure of $S^+_r(\widetilde \g(\tilde{t}))$ at some point $z$. Now there are two possibilities: either $z\in\Om$, or $z\in \de \Om$. In the first case we obtain a contradiction with the minimality of $D_{\tilde{t}}$, again by Proposition~\ref{prop:properties}(i) (indeed, $z$ would represent a non-admissible singularity for $\de D_{\tilde{t}}\cap \Om$). In the second case, we can ``cut the cusp'' formed by the two connected components at $z$ and obtain a competitor $D'$ of $D_{\tilde{t}}$ such that $\cF_\k[D'] < \cF_\k[D_{\tilde{t}}]$, which is again a contradiction. This shows that $E^m_\k = A_0$.\par

\underline{Step (iii).}
Let now $\k=h_\Om$; just as before one sees that $A_0$ is a minimal minimizer, which we shall now denote by $E^m_\k$. We need to show that it is the unique one. Let $F^m_\k$ be another minimal minimizer, i.e.~$F^m_\k \cap E^m_\k$ is empty. By Proposition~\ref{prop:palla_interna}~(iii), there exist two balls of radius $r=\k^{-1}$, $B_1 \subset E^m_\k$ and $B_2 \subset F^m_\k$. The assumption of no necks of radius $r$ grants us the existence of a $\textrm{C}^{1,1}$ curve $\widetilde \g$ in $\Om^r$ from the center of $B_1$ to that of $B_2$.  Arguing as in Step~(ii), we could construct a minimizer with a singular point in the interior boundary, which is again a contradiction.
\end{proof}

\begin{rem}\label{rem:multiple_families}
In Step~(i) of the above proof, we show that we have a one-parameter family of minimizers $\{A_t\}_t$ which ``interpolates'' between $E^m_\kappa$ and $E^M_\kappa$. Notice that this is not the only way to ``grow'' $E^m_\kappa$ into $E^M_\kappa$ but there are infinitely many as soon as $\#\G^1_r >1$. Labelling the curves $\g\in \G^1_r$ with indexes $1,\dots, n$, we let ${\boldsymbol \theta}(t) = (\theta_1(t),\dots, \theta_n(t))$ with $\theta_i(t)$ nondecreasing, surjective functions of $t$ from $[0,1]$ in $[0,1]$. Then, one can define the multi-parameter family $\{A_{{\boldsymbol \theta}(t)}\}_t$ for $t\in[0,1]$ as
\[
A_{{\boldsymbol \theta}(t)} = \left(\overline{\interior(\Om^r)} \cup \bigcup_{\g \in \Gamma^2_r} \g \cup \bigcup_{i=1}^n \gamma_i([0,\theta_i(t)])  \right) \oplus B_r,
\]
and check that these are all minimizers, by reasoning as in the proof of Lemma~\ref{lem:reach} and Step~(i) of the proof of Theorem~\ref{thm:main}.
\end{rem}

\begin{cor}\label{cor:improvement}
Let $\Om$ be a Jordan domain with $|\de \Om|=0$, and let $r$ be the unique positive solution of $\pi \rho^2 = |\Om^\rho|$. If $\Om$ has no necks of radius $r$, then, $h_\Om = r^{-1}$.
\end{cor}

\begin{proof}
We start noticing that there is a unique positive $r$ such that  the equality $\pi r^2 = |\Om^r|$ holds. This immediately follows from the fact that $\pi \rho^2$ is continuous and strictly increasing, while $|\Om^\rho|$ is continuous and decreasing. By hypothesis $\Om$ has no necks of radius $r$, and therefore $\Omega^r$ is path-connected. Moreover, as $\Om$ is simply connected, it is easy to see that $\Om^r$ is as well. Recall that by~\cite[Lemma~5.1]{LNS17} this implies that $\Om^r$ has reach at least $r$. Therefore, by Steiner's formulas we have
\[
P(\Om^r \oplus B_r) = \mathcal{M}_o(\Om^r) + 2\pi r, \qquad  |\Om^r \oplus B_r|= |\Om^r| + r \mathcal{M}_o(\Om^r) + \pi r^2.
\]
The hypothesis $|\Om^r|=\pi r^2$, paired with the above equalities, implies that
\begin{equation}\label{eq:contradiction_h=k}
P(\Om^r \oplus B_r) - \frac 1r |\Om^r \oplus B_r|=0.
\end{equation}
Therefore, the Cheeger constant of $\Om$ is bounded from above by $r^{-1}$. By Theorem~\ref{thm:main}, we immediately find that the set $\Om^r \oplus B_r$ minimizes the prescribed curvature functional $\mathcal{F}_{r^{-1}}$. Then, argue by contradiction and suppose that $h_\Om<r^{-1}$. As $\min \mathcal{F}_\k<0$ for $\k>h_\Om$, we get
\[
P(\Om^r \oplus B_r) - \frac{1}{r} |\Om^r \oplus B_r| < 0,
\]
against~\eqref{eq:contradiction_h=k}.
\end{proof}

\begin{cor}\label{cor:nestedness}
Let $\Om$ be a Jordan domain such that  $|\de \Om|=0$ and let $\k_2>\k_1\ge h_\Om$. If $\Om$ has no necks of radius $\k_2^{-1}$ and $\k_1^{-1}$, then one has
\[
E^M_{\k_2} \supseteq E^m_{\k_2}  \supseteq E^M_{\k_1} \supseteq E^m_{\k_1}.
\]
\end{cor}
\begin{proof}
Let $r_i = \k_i^{-1}$ for $i=1,2$. Since $r_2<r_1$, the set $\overline{\interior(\Om^{r_2})}$ contains $\Om^{r_1}$, hence $\overline{\interior(\Om^{r_2})}\oplus B_{r_2}$ contains $\Om^{r_1}\oplus B_{r_1}$. Then the proof directly follows from Theorem~\ref{thm:main}.
\end{proof}

\begin{rem}\label{rem:nestedness}
Notice that the set inclusion $E^m_{\k_2}  \supseteq E^M_{\k_1}$ is strict as soon as $|\Om|>|E^M_{\k_1}|$. Indeed, this strict volume bound implies that $\de E^M_{\k_1} \cap \Om$ is not empty. Assume by contradiction that $E^m_{\k_2} = E^M_{\k_1}$. Then, we infer that the interior boundary $\de E^M_{\k_1} \cap \Om = \de E^M_{\k_2} \cap \Om$, which is not empty, must have curvature equal to both $\k_1$ and $\k_2$, which is not possible.
\end{rem}

\begin{rem}\label{rem:uniqueness_no_necks_r}
If one assumes that $\Om$ is a Jordan domain with  $|\de \Om|=0$ and that has no necks of radius $r$ for all $r$, then the solution of~\eqref{eq:pmc} is unique for almost every $\k\ge h_\Om$, i.e.~there are at most countably many $\k$ for which uniqueness does not hold. For the sake of completeness, we remark that this is equivalent to say that there are at most countably many $r\le \inr(\Om)$ such that $\Gamma^1_r$ is not empty. To see this, let us set
\[
V_\k = |E^M_\k| - |E^m_\k| = |E^M_\k \setminus E^m_\k|.
\]
For all $\k$ such that uniqueness does not hold, one has $V_\k >0$. At the same time Corollary~\ref{cor:nestedness} implies that for $\k_2>\k_1$ the sets $E^M_{\k_1} \setminus E^m_{\k_1}$ and $E^M_{\k_2} \setminus E^m_{\k_2}$ are pairwise disjoint. Therefore, there are at most countably many $\k$ such that $V_\k >0$.  An example of set admitting countably many values $\k$ such that uniqueness fails is the ``ziggurat'' in Figure~\ref{fig:ziggurat}, built as follows. First, let $Q$ be the square $Q=[-2^{-1}, 2^{-1}]\times[0,1]$ and let $f(n)$ be the sequence
\begin{align*}
f(1)= \frac 12, && f(n)=\frac 12 + \sum_{i=2}^n \frac{1}{2^{i-1}}, \quad \forall n>1.
\end{align*}
Let then $Q^+_n$ and $Q^-_n$ be the squares
\begin{align*}
Q^+_n &=  [f(n), f(n+1)]\times [0, 2^{-n}], \\
Q^-_n &= [-f(n+1), -f(n)]\times [0, 2^{-n}].
\end{align*}
We define the ziggurat as (the interior of)
\[
Q\cup \bigcup_{n\ge 1} (Q^+_n \cup Q^-_n).
\] 
The resulting set has no necks of radius $r$ for all $r\le \inr(\Om)$ and it is such that $\G^1_r\neq \emptyset$ whenever $r=2^{-n-1}$ for any $n\in \N$. Therefore, uniqueness of minimizers of $\cF_\kappa$ fails whenever $\k=2^{n+1}\ge h_\Om$.\par
Similar examples are given by suitable fractals, e.g.~by a \emph{square Koch snowflake}, i.e.~the set obtained replacing the sides of a unit square with suitable quadratic type $1$ Koch curves (e.g.~by iteratively replacing each middle $n$-th part of a segment with a square, with $n>3$).
\end{rem}

\begin{figure}[t]
\centering
\includegraphics[trim={0 3cm 0 3cm},clip, width=.75\textwidth]{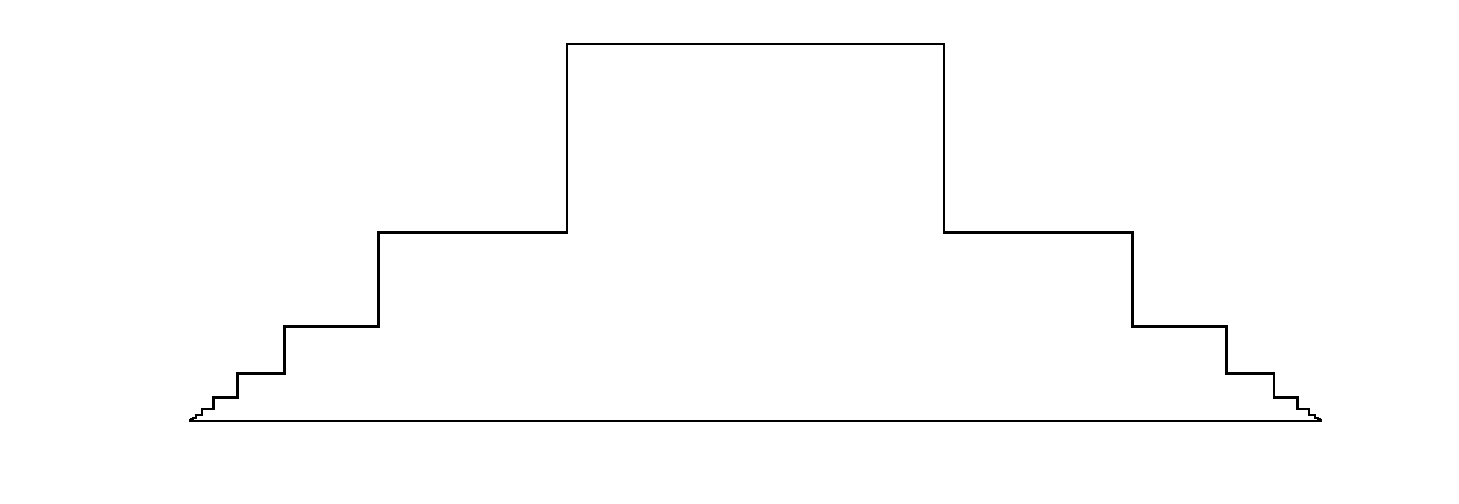} 
\caption{A ziggurat:~a set with no necks of radius $r$ for all $r$ such that it has infinitely (uncountably) many solutions to the prescribed curvature equation relative to infinitely (countably) many values of $\k$.}
\label{fig:ziggurat}
\end{figure}

\section{Proof of Theorem~\ref{thm:m_iso}}\label{sec:iso}

In this section we exploit Theorem~\ref{thm:main} to describe minimizers of the isoperimetric profile $\mathcal{J}$, relatively to a Jordan domain $\Om$, with $|\de \Om|=0$ and such that it has no necks of radius $r$ for all $r\le h_\Om^{-1}$. In particular we shall show that for any volume $V$ greater than or equal to $|E^m_{h_\Om}|$, there exists a suitable $\k$ and a suitable minimizer $E_\k$ of $\cF_\k$ such that $|E_\k|=V$, hence $\mathcal{J}(V) = P(E_\k)$. As a consequence, we are able to prove that for that class of $\Om$ the isoperimetric profile is convex for $V\ge |E^m_{h_\Om}|$, by showing that it is the Legendre transform of $\mathcal{G}\colon \k \mapsto -\min \cF_\k$, defined for $\k \ge h_\Om$. Trivially, $\mathcal J$ (resp.~$\mathcal{J}^2$) is concave (resp.~convex) for $V\le |B_R|$, where $R=\inr(\Om)$; it would be of interest managing to prove that $\mathcal{J}^2$ is convex on the whole range $[0, |\Om|]$, which up to our knowledge has not been addressed when considering the total perimeter $P(E)$. For the sake of completeness, we recall that the square of the \emph{relative} isoperimetric profile (i.e.~of the infimum of the relative perimeter $P(E;\Om)$ under volume constraint $V$) is known to be concave in convex bodies (see~\cite{Kuw03} and~\cite{LRV18}).\par

In order to prove such a theorem we need first the following technical lemma which ensures the semicontinuity of the outer Minkowski content of $\Om^r$ and $\overline{\interior(\Om^r)}$, whenever $\Om$ has no necks of radius $r$ for all $r<R$, for a fixed $R>0$.

\begin{lem}\label{lem:upper_lower_minkowski}
Let $R>0$ be fixed and let $\Om$ be a Jordan domain with no necks of radius $r$ for all $r<R$. Then the functions
\begin{align*}
m(r) = \mathcal{M}_o(\Om^r),\qquad
\mu(r) = \mathcal{M}_o(\overline{\interior(\Om^r)})
\end{align*}
are, respectively, upper semicontinuous and lower semicontinuous on $(0,R)$.
\end{lem}
\begin{proof}
We assume without loss of generality that $\Om^r$ is not empty for all $0<r<R$. By~\cite[Lemma~5.1]{LNS17} we know that $\reach(\Om^r) \ge r$. Moreover, $\Om^r$ is simply connected, hence by Steiner's formulas we have for all $0<\e<r$
\[
|\Om^r\oplus B_\e| = |\Om^r| + \e m(r) + \pi \e^2,
\]
whence
\begin{equation}\label{eq:formulam}
m(r) = \frac{|\Om^r\oplus B_\e| - |\Om^r|}{\e} + \pi \e.
\end{equation}
Fix now $r_0\in (0,R)$ and $0<\e<r_0$. Thanks to~\eqref{eq:formulam}, the upper semicontinuity of $m(r)$ at $r_0$ follows as soon as we prove that 
\[
\limsup_{r\to r_0} \alpha_\e(r,r_0) + \beta(r,r_0) \le 0,
\]
where we have set
\[
\alpha_\e(r,r_0) = |\Om^r\oplus B_\e| - |\Om^{r_0} \oplus B_\e|,\qquad 
\beta(r,r_0) = \alpha_0(r,r_0) = |\Om^r| - |\Om^{r_0}|.
\]
The fact that $\limsup_{r\to r_0} \beta(r,r_0) \le 0$ immediately follows from the following simple observations. First, as $r\to r_0^-$, we have $\Om^r\supset \Om^{r_0}$ and $|\Om^r\setminus \Om^{r_0}|\to 0$, so that in particular $|\Om^r|\to |\Om^{r_0}|$. Second, as $r\to r_0^+$, we have $\Om^r\subset \Om^{r_0}$ and thus $|\Om^r|\le |\Om^{r_0}|$. We are left with showing $\limsup_{r\to r_0} \alpha_\e(r,r_0) \le 0$. We first consider the case of the left upper limit, i.e.~when $r\to r_0^-$. In this case we have $\Om^{r_0}\oplus B_\e \subset \Om^r \oplus B_\e$ by monotonicity. Moreover,
\begin{equation}\label{eq:limOmreps}
\lim_{r\uparrow r_0}\, (\Om^r\oplus B_\e)\setminus (\Om^{r_0}\oplus B_\e) = \bigcap_{r<r_0} \big(\Om^r\oplus B_\e\big)\, \setminus (\Om^{r_0}\oplus B_\e) 
\subset \de (\Om^{r_0}\oplus B_\e),
\end{equation}
where to prove the last inclusion one can rely on the fact that $\Om^r$ converges to $\Om^{r_0}$ w.r.t.~the Hausdorff distance as $r\to r_0^-$. Since  $\e< r_0$ the set $\de(\Om^{r_0}\oplus B_\e)$ is Lipschitz, thus it has zero Lebesgue measure. Hence by~\eqref{eq:limOmreps}, we find
\[
\limsup_{r\to r_0^-} \alpha_\e(r,r_0) \le |\de(\Om^{r_0}\oplus B_\e)| = 0.
\]
Concerning the right upper limit, i.e.~when $r\to r_0^+$, we simply observe that $\alpha_\e(r,r_0)\le 0$ whenever $r>r_0$ by monotonicity, hence a fortiori we obtain the desired $\limsup$ inequality. This completes the proof of the upper semicontinuity of $m(r)$.\par

We now set $W^r = \overline{\interior(\Om^r)}$. By Remark~\ref{rem:nogammadue} one has $\G^2_r=\emptyset$, thus by Lemma~\ref{lem:reach} we know that $W^r$ is simply connected and $\reach(W^r)\ge r$. If we denote by $\xi: \overline{W^r\oplus B_{r/2}} \to W^r$ the unique projection map onto $W^r$ (which is well defined by the reach property of $W^r$), we infer by~\cite[Theorem~4.8]{FedererCM} that the restriction of $\xi$ to the boundary of $W^r\oplus B_{r/2}$ is $2$-Lipschitz and its image is the boundary of $W^r$. This shows that $\de W^r$ is the Lipschitz image of a smooth curve with finite length (indeed, the boundary of $W^r\oplus B_{r/2}$ is of class $\textrm{C}^{1,1}$ with bounded curvature). Consequently, we have that $\mu(r) = \H^1(\de W^r) = P(W^r)$, where the first equality follows from~\cite[Section~3.2.39]{FedererBOOK} and the second from the fact that $\de W^r$ is a continuous curve. At the same time, the mapping $r\mapsto W^r$ is continuous with respect to the $L^1$-topology, which can be proved by observing that $W^r$ is Lebesgue equivalent to both $\interior(\Om^r)$ (a consequence of $\H^1(\de W^r)<+\infty$) and $\Om^r$ (a consequence of Proposition~\ref{prop:struttura_diff}). We thus conclude that $\mu(r)$ is lower semicontinuous on $(0,R)$.
\end{proof}

We are now ready to prove the result concerning the minimizers of the isoperimetric profile. Since we know by Theorem~\ref{thm:main} the structure of $E^m_\k$ and of $E^M_\k$, it is easy to show that there is a family of minimizers with volumes spanning the above range. Actually, this has already been done in Step~(i) of Theorem~\ref{thm:main}, where we proved that the sets
\[
A_t = \left(\overline{\interior(\Om^r)} \cup \bigcup_{\gamma \in \Gamma^2_r} \gamma \cup \bigcup_{\gamma \in \G^1_r} \gamma([0,t])  \right) \oplus B_r\,,
\]
with $r=\k^{-1}$, are minimizers. It is straightforward that $|A_t|$ is a continuous function in $t$ from $[0,1]$ to $[|E^m_\k|, |E^M_\k|]$. Thus, the proof boils down to show that there are no volume gaps when increasing the value of $\k$, and this is exactly what the previous lemma is needed for.

\begin{proof}[Proof of Theorem~\ref{thm:m_iso}]
We can directly assume that $|\Om|>|E^m_{h_\Om}|$, otherwise there is nothing to prove. If $V$ equals either $|\Om|$ or $|E^m_{h_\Om}|$, there is nothing to prove. Then, let $|\Om|>V>|E^m_{h_\Om}|$ and set $\k^*$ and $\k_*$ as follows
\[
\k^*=\inf\{\,\k: |E^M_\k|>V\,\}, \qquad \qquad \k_*=\sup\{\,\k: |E^m_\k|<V\,\}.
\]
Since the functions $\k\mapsto |E^M_\k|$ and $\k\mapsto |E^m_\k|$ are nondecreasing, and for all $\k$ one has $|E^M_\k|\ge |E^m_\k|$, the inequality $\k^*\ge \k_*$ obviously holds. We first show that these two values agree. Suppose that $\k^*>\k_*$. Then, for any $\k \in (\k_*, \k^*)$ one can consider the minimizers of $\cF_\k$. On the one hand, the lower bound $\k > \k_*$ implies $|E^m_\k| \ge V$, while the upper bound $\k < \k^*$ implies $|E^M_\k| \le V$. As $|E^M_\k| \le |E^m_\k|$, we immediately find that $|E^M_\k|=|E^m_\k|=V$ for all $\k \in (\k_*, \k^*)$. The strict nestedness granted by Corollary~\ref{cor:nestedness} and Remark~\ref{rem:nestedness} immediately yields a contradiction.\par

Hence, $\k^*=\k_*=\hat \k$.  Setting $\hat r=1/\hat \k$, by Steiner's formulas we have that
\begin{align*}
|E^M_{\hat \k}|=\pi {\hat r}^2 + {\hat r}\mathcal{M}_o(\Om^{\hat r}) + |\Om^{\hat r}|\,, && |E^m_{\hat \k}|=\pi {\hat r}^2 + {\hat r}\mathcal{M}_o(\overline{\interior(\Om^{\hat r})}) + |\overline{\interior(\Om^{\hat r})}|\,.
\end{align*}
Therefore, by Lemma~\ref{lem:upper_lower_minkowski} we get that
\[
|E^M_{\hat\k}| \ge V\ge |E^m_{\hat \k}|\,.
\]
If either one of the two inequalities is not strict, we are done. Hence, suppose that both are strict. This implies that $\Gamma^1_{\hat r}$ is not empty. The sets $A_t$ in~\eqref{eq:def_A_t} give a family of minimizers with volume increasing continuously from $|E^m_{\hat \k}|$ up to  $|E^M_{\hat \k}|$, thus one finds a suitable $t$ such that $|A_t|=V$, as required.
\end{proof}

\begin{prop}\label{prop:legendre}
Let $\Om$ be a Jordan domain such that  $|\de \Om|=0$. Assume that $\Om$ has no necks of radius $r$, for all $r\le h_\Om^{-1}$. Then, the isoperimetric profile $\mathcal J$ for $V\ge |E^m_{h_\Om}|$ is the Legendre transform of $\mathcal{G}\colon \k\mapsto -\min\cF_\k$ defined for $\k \ge h_\Om$.
\end{prop}

\begin{proof}
First, we prove that the map $\mathcal{G}$ is convex. Set $\k_1, \k_2 \ge h_\Om$ and consider the convex combination $\hat \k = t\k_1 +(1-t)\k_2$. As usual we denote by $E_\k$ a minimizer of $\cF_\k$. Then, one has
\begin{align*}
\mathcal{G}(t\k_1 +(1-t)\k_2) &= \mathcal{G}(\hat \k) = -\min \cF_{\hat \k} = -P(E_{\hat \k}) +\hat \k|E_{\hat \k}|\\
&=-t(P(E_{\hat \k}) - \k_1|E_{\hat \k}|) -(1-t)(P(E_{\hat \k}) - \k_2|E_{\hat \k}|)\\
&\le -t \min \cF_{\k_1} -(1-t)\min \cF_{\k_2}\\
&= t\mathcal{G}(\k_1) +(1-t)\mathcal{G}(\k_2).
\end{align*}
Therefore, one can consider the Legendre transform of $\mathcal{G}$. By definition we have
\[
\mathcal{G}^*(V) = \sup_{\k \ge h} \{\,\k V -\mathcal{G}(\k)\,\} = \sup_{\k \ge h} \big\{\,\k V +\min\{P(E)-\k|E|\,\} \big\}\,.
\]
By Theorem~\ref{thm:m_iso} for all $V\ge |E^m_{h_\Om}|$ there exist $\bar\k\ge h_\Om$ and $E_{\bar \k}$ minimizer of $\cF_{\bar \k}$ with $|E_{\bar \k}|=V$ and such that $\mathcal{J} (V)= P(E_{\bar \k})$. Hence, on the one hand
\begin{align*}
\mathcal{G}^*(V) &\ge \bar \k V + \min\{\,P(E)-\bar \k|E|\,\} \\
&= \bar \k V + P(E_{\bar \k}) - \bar \k |E_{\bar \k}|=  P(E_{\bar \k}) = \mathcal{J} (V).
\end{align*}
On the other hand, for all $\k$ we have
\[
\k V -\mathcal{G}(\k) = \k V +\min \cF_\k \le \k V + P(E_{\bar \k}) -\k |E_{\bar \k}| = P(E_{\bar \k}).
\]
Thus,
\begin{align*}
\mathcal{G}^*(V) &\le \sup_{\k \ge h} \{\,\k V + P(E_{\bar \k})- \k|E_{\bar \k}|\,\} \\
&= \sup_{\k \ge h} \{\,P(E_{\bar \k})\,\}=  P(E_{\bar \k}) = \mathcal{J} (V),
\end{align*}
and the claim follows at once.
\end{proof}

Since the Legendre transform maps convex functions in convex functions, one has the following corollary.

\begin{cor}\label{cor:convexity}
Let $\Om$ be a Jordan domain such that  $|\de \Om|=0$. If $\Om$ has no necks of radius $r$ for all $r\in (0, h_\Om^{-1}]$, then the isoperimetric profile $\mathcal J$ is convex in $[|E^m_{h_\Om}|, |\Om|]$.
\end{cor}

\begin{rem}
Notice that whenever $\G^1_r\neq \emptyset$, $\mathcal J$ is linear in the interval of volumes delimited by $|E^m_{r^{-1}}|$ and $|E^M_{r^{-1}}|$. There are sets with no necks of radius $r$ for all $r$ that display such linear growth on countably many intervals of smaller and smaller size. Think of the ziggurat described in Remark~\ref{rem:uniqueness_no_necks_r} and shown in Figure~\ref{fig:ziggurat}. Moreover, Remark~\ref{rem:multiple_families} shows that nestedness of minimizers of the isoperimetric profile is not ensured --- even though a nested family can always be chosen. This is achieved, for instance, by interpolating between $E^m_\k$ and $E^M_\k$, always through the family $\{A_t\}_t$ defined in~\eqref{eq:def_A_t}.
\end{rem}

\section*{Conflict of Interest}
The authors declare that they have no conflict of interest.

\section*{Acknowledgements}
The authors would like to thank Hugo Lavenant, Manuel Ritor\'e, and Aldo Pratelli for fruitful conversations and for their comments on a preliminary version of the paper.

\bibliographystyle{plainurl}

\bibliography{pmc_no_necks}

\begin{thebibliography}{10}

\bibitem{ACC05}
F.~Alter, V.~Caselles, and A.~Chambolle.
\newblock A characterization of convex calibrable sets in $\mathbb{R}^n$.
\newblock {\em Math. Ann.}, 332(2):329--366, 2005.
\newblock \href {https://doi.org/10.1007/s00208-004-0628-9}
  {\path{doi:10.1007/s00208-004-0628-9}}.

\bibitem{ABCMP17}
A.~Alvino, F.~Brock, F.~Chiacchio, A.~Mercaldo, and M.~R. Posteraro.
\newblock Some isoperimetric inequalities on $\mathbb{R}^n$ with respect to
  weights $|x|^\alpha$.
\newblock {\em J. Math. Anal. Appl.}, 451(1):280--318, 2017.
\newblock \href {https://doi.org/10.1016/j.jmaa.2017.01.085}
  {\path{doi:10.1016/j.jmaa.2017.01.085}}.

\bibitem{ABCMP19a}
A.~Alvino, F.~Brock, F.~Chiacchio, A.~Mercaldo, and M.~R. Posteraro.
\newblock On weighted isoperimetric inequalities with non-radial densities.
\newblock {\em Appl. Anal.}, 98(10):1935--1945, 2019.
\newblock \href {https://doi.org/10.1080/00036811.2018.1506106}
  {\path{doi:10.1080/00036811.2018.1506106}}.

\bibitem{ACMM01}
L.~Ambrosio, V.~Caselles, S.~Masnou, and J.-M. Morel.
\newblock Connected components of sets of finite perimeter and applications to
  image processing.
\newblock {\em J. Eur. Math. Soc. (JEMS)}, 3(1):39--92, 2001.
\newblock \href {https://doi.org/10.1007/PL00011302}
  {\path{doi:10.1007/PL00011302}}.

\bibitem{AFP00book}
L.~Ambrosio, N.~Fusco, and D.~Pallara.
\newblock {\em Functions of {B}ounded {V}ariation and {F}ree {D}iscountinuity
  {P}roblems}.
\newblock Oxford University Press, 2000.

\bibitem{BGM03}
E.~Barozzi, E.~Gonzalez, and U.~Massari.
\newblock The mean curvature of a {L}ipschitz continuous manifold.
\newblock {\em Rend. Mat. Acc. Lincei}, 14(4):257--277, 2003.
\newblock URL: \url{http://www.bdim.eu/item?id=RLIN_2003_9_14_4_257_0}.

\bibitem{BCG02}
F.~Bethuel, P.~Caldiroli, and M.~Guida.
\newblock Parametric surfaces with prescribed mean curvature.
\newblock {\em Rend. Sem. Mat. Univ. Torino}, 60(4):175--231, 2002.
\newblock URL:
  \url{http://www.seminariomatematico.unito.it/rendiconti/60-4.html}.

\bibitem{BS03}
G.~Buttazzo and E.~Stepanov.
\newblock Optimal transportation networks as free {D}irichlet regions for the
  {M}onge--{K}antorovich problem.
\newblock {\em Ann. Sc. Norm Super. Pisa, Cl. Sci. (5)}, 2(4):631--678, 2003.

\bibitem{CRO13}
X.~Cabr\'e and X.~Ros-Oton.
\newblock Sobolev and isoperimetric inequalities with monomial weights.
\newblock {\em J. Differential Equations}, 255:4312--4336, 2013.
\newblock \href {https://doi.org/10.1016/j.jde.2013.08.010}
  {\path{doi:10.1016/j.jde.2013.08.010}}.

\bibitem{CMN10}
V.~Caselles, M.~Jr. Miranda, and M.~Novaga.
\newblock Total variation and {C}heeger sets in {G}auss space.
\newblock {\em J. Funct. Anal.}, 259(6):1491--1516, 2010.
\newblock \href {https://doi.org/10.1016/j.jfa.2010.05.007}
  {\path{doi:10.1016/j.jfa.2010.05.007}}.

\bibitem{Che80}
J.~T. Chen.
\newblock On the existence of capillary free surfaces in the absence of
  gravity.
\newblock {\em Pacific J. Math.}, 88(2):323--361, 1980.
\newblock \href {https://doi.org/10.2140/pjm.1980.88.323}
  {\path{doi:10.2140/pjm.1980.88.323}}.

\bibitem{CF15}
G.~Crasta and I.~Fragal\`a.
\newblock On the {D}irichlet and {S}errin problems for the inhomogeneous
  infinity {L}aplacian in convex domains: regularity and geometric results.
\newblock {\em Arch. Rational Mech. Anal.}, 218(3):1577--1607, 2015.
\newblock \href {https://doi.org/10.1007/s00205-015-0888-4}
  {\path{doi:10.1007/s00205-015-0888-4}}.

\bibitem{Csa15}
G.~Csat\'o.
\newblock An isoperimetric problem with density and the {H}ardy {S}obolev
  inequality in $\mathbb{R}^2$.
\newblock {\em Diff. Int. Equations}, 28(9--10):971--988, 2015.

\bibitem{FLSS18}
D.~De~Ford, H.~Lavenant, Z.~Schutzman, and J.~Solomon.
\newblock Total variation isoperimetric profiles.
\newblock {\em SIAM J. Appl. Algebra Geom.}, 3(4):585--613, 2019.
\newblock \href {https://doi.org/10.1137/18M1215943}
  {\path{doi:10.1137/18M1215943}}.

\bibitem{FedererCM}
H.~Federer.
\newblock Curvature measures.
\newblock {\em Trans. Amer. Math. Soc.}, 93(3):418--491, 1959.
\newblock \href {https://doi.org/10.2307/1993504} {\path{doi:10.2307/1993504}}.

\bibitem{FedererBOOK}
H.~Federer.
\newblock {\em Geometric {M}easure {T}heory}, volume 153 of {\em Die
  Grundlehren der mathematischen Wissenschaften}.
\newblock Springer-Verlag New York Inc., New York, 1969.

\bibitem{FG79a}
R.~Finn and E.~Giusti.
\newblock {Existence and non existence of capillary surfaces}.
\newblock {\em Manuscripta Math.}, 28(1--3):1--11, 1979.
\newblock \href {https://doi.org/10.1007/BF01647961}
  {\path{doi:10.1007/BF01647961}}.

\bibitem{FG79b}
R.~Finn and E.~Giusti.
\newblock Nonexistence and existence of capillary surfaces.
\newblock {\em Manuscripta Math.}, 28(1--3):13--20, 1979.
\newblock \href {https://doi.org/10.1007/BF01647962}
  {\path{doi:10.1007/BF01647962}}.

\bibitem{FK02}
R.~Finn and A.~A.~Jr. Kosmodem'yanskii.
\newblock Some unusual comparison properties of capillary surfaces.
\newblock {\em Pacific J. Math.}, 205(1):119--137, 2002.
\newblock \href {https://doi.org/10.2140/pjm.2002.205.119}
  {\path{doi:10.2140/pjm.2002.205.119}}.

\bibitem{Gia73}
M.~Giaquinta.
\newblock Regolarit\`a delle superfici {$BV(\Omega)$} con curvature media
  assegnata.
\newblock {\em Boll. Un. Mat. Ital. (4)}, 8:567--578, 1973.

\bibitem{Gia74}
M.~Giaquinta.
\newblock On the {D}irichlet problem for surfaces of prescribed mean curvature.
\newblock {\em Manuscripta Math.}, 12:73--86, 1974.
\newblock \href {https://doi.org/10.1007/BF01166235}
  {\path{doi:10.1007/BF01166235}}.

\bibitem{Giu78}
E.~Giusti.
\newblock On the equation of surfaces of prescribed mean curvature. {E}xistence
  and uniqueness without boundary conditions.
\newblock {\em Invent. Math.}, 46(2):111--137, 1978.
\newblock \href {https://doi.org/10.1007/BF01393250}
  {\path{doi:10.1007/BF01393250}}.

\bibitem{GN12}
M.~Goldman and M.~Novaga.
\newblock Volume-constrained minimizers for the prescribed curvature problem in
  periodic media.
\newblock {\em Calc. Var. Partial Differential Equations}, 44(3--4):297--318,
  2012.
\newblock \href {https://doi.org/10.1007/s00526-011-0435-6}
  {\path{doi:10.1007/s00526-011-0435-6}}.

\bibitem{GMT81}
E.~Gonzalez, U.~Massari, and I.~Tamanini.
\newblock Minimal boundaries enclosing a given volume.
\newblock {\em Manuscripta Math.}, 34(2--3):381--395, 1981.
\newblock \href {https://doi.org/10.1007/BF01165546}
  {\path{doi:10.1007/BF01165546}}.

\bibitem{GR10}
M.~Guida and S.~Rolando.
\newblock Symmetric $\kappa$-loops.
\newblock {\em Diff. Int. Equations}, 23:861--898, 2010.

\bibitem{Hu91}
X.~Huang.
\newblock Closed surface with prescribed mean curvature in $\mathbb{R}^3$.
\newblock {\em Science in China}, 34(10), 1991.
\newblock \href {https://doi.org/10.1360/ya1991-34-10-1162}
  {\path{doi:10.1360/ya1991-34-10-1162}}.

\bibitem{KL06}
B.~Kawohl and T.~Lachand-Robert.
\newblock Characterization of {C}heeger sets for convex subsets of the plane.
\newblock {\em Pacific J. Math.}, 225(1):103--118, 2006.
\newblock \href {https://doi.org/10.2140/pjm.2006.225.103}
  {\path{doi:10.2140/pjm.2006.225.103}}.

\bibitem{KP11}
D.~Krej\v{c}i\v{r}\'ik and A.~Pratelli.
\newblock The {C}heeger constant of curved strips.
\newblock {\em Pacific J. Math.}, 254(2):309--333, 2011.
\newblock \href {https://doi.org/10.2140/pjm.2011.254.309}
  {\path{doi:10.2140/pjm.2011.254.309}}.

\bibitem{Kuw03}
E.~Kuwert.
\newblock {\em Geometric Analysis and Nonlinear Partial Differential
  Equations}, chapter Note on the isoperimetric profile of a convex body, pages
  195--200.
\newblock Springer, Berlin, Heidelberg, 2003.
\newblock \href {https://doi.org/10.1007/978-3-642-55627-2_12}
  {\path{doi:10.1007/978-3-642-55627-2_12}}.

\bibitem{Leo15}
G.~P. Leonardi.
\newblock An overview on the {C}heeger problem.
\newblock In {\em New {T}rends in {S}hape {O}ptimization}, volume 166 of {\em
  Internat. Ser. Numer. Math.}, pages 117--139. Springer Int. Publ., 2015.
\newblock \href {https://doi.org/10.1007/978-3-319-17563-8_6}
  {\path{doi:10.1007/978-3-319-17563-8_6}}.

\bibitem{LNS17}
G.~P. Leonardi, R.~Neumayer, and G.~Saracco.
\newblock The {C}heeger constant of a {J}ordan domain without necks.
\newblock {\em Calc. Var. Partial Differential Equations}, 56:164, 2017.
\newblock \href {https://doi.org/10.1007/s00526-017-1263-0}
  {\path{doi:10.1007/s00526-017-1263-0}}.

\bibitem{LP16}
G.~P. Leonardi and A.~Pratelli.
\newblock On the {C}heeger sets in strips and non--convex domains.
\newblock {\em Calc. Var. Partial Differential Equations}, 55(1):15, 2016.
\newblock \href {https://doi.org/10.1007/s00526-016-0953-3}
  {\path{doi:10.1007/s00526-016-0953-3}}.

\bibitem{LRV18}
G.~P. Leonardi, M.~Ritor\'e, and E.~Vernadakis.
\newblock Isoperimetric inequalities in unbounded convex bodies.
\newblock To appear in \emph{Mem. Amer. Math. Soc.}

\bibitem{LS18a}
G.~P. Leonardi and G.~Saracco.
\newblock The prescribed mean curvature equation in weakly regular domains.
\newblock {\em NoDEA Nonlinear Differ. Equ. Appl.}, 25(2):9, 2018.
\newblock \href {https://doi.org/10.1007/s00030-018-0500-3}
  {\path{doi:10.1007/s00030-018-0500-3}}.

\bibitem{LS18b}
G.~P. Leonardi and G.~Saracco.
\newblock Two examples of minimal {C}heeger sets in the plane.
\newblock {\em Ann. Mat. Pura Appl. (4)}, 197(5):1511--1531, 2018.
\newblock \href {https://doi.org/10.1007/s10231-018-0735-y}
  {\path{doi:10.1007/s10231-018-0735-y}}.

\bibitem{Mag12book}
F.~Maggi.
\newblock {\em Sets of {F}inite {P}erimeter and {G}eometric {V}ariational
  {P}roblems}, volume 135 of {\em Cambridge Studies in Advanced Mathematics}.
\newblock Cambridge University Press, Cambridge, 2012.
\newblock \href {https://doi.org/10.1017/CBO9781139108133}
  {\path{doi:10.1017/CBO9781139108133}}.

\bibitem{Mas74}
U.~Massari.
\newblock Esistenza e regolarit\`a delle ipersuperfici di curvatura media
  assegnata in {$\mathbb{R}^n$}.
\newblock {\em Arch. Rational Mech. Anal.}, 55:357--382, 1974.
\newblock \href {https://doi.org/10.1007/BF00250439}
  {\path{doi:10.1007/BF00250439}}.

\bibitem{MR10}
F.~Morgan and A.~Ros.
\newblock Stable constant-mean-curvature hypersurfaces are area minimizing in
  small {$L^1$} neighborhoods.
\newblock {\em Interfaces Free Bound.}, 12(2):151--155, 2010.
\newblock \href {https://doi.org/10.4171/IFB/230} {\path{doi:10.4171/IFB/230}}.

\bibitem{Par11}
E.~Parini.
\newblock An introduction to the {C}heeger problem.
\newblock {\em Surv. Math. Appl.}, 6:9--21, 2011.
\newblock URL: \url{http://www.utgjiu.ro/math/sma/v06/v06.html}.

\bibitem{PS17}
A.~Pratelli and G.~Saracco.
\newblock On the generalized {C}heeger problem and an application to 2d strips.
\newblock {\em Rev. Mat. Iberoam.}, 33(1):219--237, 2017.
\newblock \href {https://doi.org/10.4171/RMI/934} {\path{doi:10.4171/RMI/934}}.

\bibitem{PS18}
A.~Pratelli and G.~Saracco.
\newblock On the isoperimetric problem with double density.
\newblock {\em Nonlinear Anal.}, 177(Part B):733--752, 2018.
\newblock \href {https://doi.org/10.1016/j.na.2018.04.009}
  {\path{doi:10.1016/j.na.2018.04.009}}.

\bibitem{PS19}
A.~Pratelli and G.~Saracco.
\newblock The $\varepsilon-\varepsilon^\beta$ property in the isoperimetric
  problem with double density, and the regularity of isoperimetric sets.
\newblock {\em Adv. Nonlinear Stud.}, Ahead of publication.
\newblock \href {https://doi.org/10.1515/ans-2020-2074}
  {\path{doi:10.1515/ans-2020-2074}}.

\bibitem{Sag94}
H.~Sagan.
\newblock {\em {Space-Filling Curves}}.
\newblock Springer New York, 1994.
\newblock \href {https://doi.org/10.1007/978-1-4612-0871-6}
  {\path{doi:10.1007/978-1-4612-0871-6}}.

\bibitem{SS19}
A.~Saracco and G.~Saracco.
\newblock A discrete districting plan.
\newblock {\em Netw. Heterog. Media}, 14(4), 2019.
\newblock \href {https://doi.org/10.3934/nhm.2019031}
  {\path{doi:10.3934/nhm.2019031}}.

\bibitem{Sar18}
G.~Saracco.
\newblock Weighted {C}heeger sets are domains of isoperimetry.
\newblock {\em Manuscripta Math.}, 156(3--4):371--381, 2018.
\newblock \href {https://doi.org/10.1007/s00229-017-0974-z}
  {\path{doi:10.1007/s00229-017-0974-z}}.

\bibitem{Sar19}
G.~Saracco.
\newblock A sufficient criterion to determine planar self-{C}heeger sets, 2019.
\newblock Preprint.
\newblock \href {http://arxiv.org/abs/1906.12101} {\path{arXiv:1906.12101}}.

\bibitem{SZ97}
E.~Stredulinsky and W.~P. Ziemer.
\newblock Area minimizing sets subject to a volume constraint in a convex set.
\newblock {\em J. Geom. Anal.}, 7(4):653--677, 1997.
\newblock \href {https://doi.org/10.1007/BF02921639}
  {\path{doi:10.1007/BF02921639}}.

\bibitem{Til10}
P.~Tilli.
\newblock Some explicit examples of minimizers for the irrigation problem.
\newblock {\em J. Convex Anal.}, 17(2), 2010.
\newblock URL: \url{http://www.heldermann.de/JCA/JCA17/JCA172/jca17039.htm}.

\bibitem{TW83}
A.~Treibergs and S.~W. Wei.
\newblock Embedded hyperspheres with prescribed mean curvature.
\newblock {\em J. Differential Geom.}, 18(3):513--521, 1983.
\newblock \href {https://doi.org/10.4310/jdg/1214437786}
  {\path{doi:10.4310/jdg/1214437786}}.

\bibitem{Yau82}
S.~T. Yau.
\newblock Problem section.
\newblock In {\em Seminar on {D}ifferential {G}eometry}, volume 102, pages
  669--706. Princeton Univ. Press, Princeton, N.J., 1982.

\end{thebibliography}

\end{document}